\documentclass[12pt,a4paper]{article} 

\usepackage{amsmath} 
\usepackage{amssymb} 
\usepackage{graphicx} 
\usepackage[utf8]{inputenc} 
\usepackage{hyperref} 
\usepackage[square,numbers,sort&compress]{natbib} 
\usepackage{blindtext} 

\usepackage{amsthm}
\usepackage{latexsym}
\usepackage{mathtools}

\renewcommand{\equation}

\newtheorem{lem}{Lemma}[section]
\newtheorem{prop}{Proposition}[section]
\newtheorem{thm}{Theorem}[section]
\newtheorem{cor}{Corollary}[section]
\newtheorem{fact}{Fact}[section]

\theoremstyle{definition}

\newtheorem{defini}{Definition}[section]
\newtheorem{rmk}{Remark}[section]
\newtheorem{example}{Example}[section]

\begin{document}

\title{An example of biquandle and an invariant of long virtual knots defined through it}
\author{Nozomu Sekino}
\date{}

\maketitle

\begin{abstract}
We introduce an example of biquandle and define an invariant of long virtual knots through it. 
As an application, we give a necessary condition for two long virtual knots commuting. 
\end{abstract}

\section{Introduction}
Biquandles are algebraic structures with axioms derived from the Reidemeister moves for oriented knots and links introduced in \cite{fenn}. 
They have been used to define invariants of oriented knots and links in so many works. 
Virtual knots were introduced by Kauffmann \cite{kauffman} as a generalization of classical knots. 
While virtual knots share many properties with classical knots, there are also many differences between them. 
For example, the resultant of the connected sum operation of fixed two virtual knots may not be unique in contrast to classical knots. 
For the case of long virtual knots, which is counterparts of classical long knots, products of two long virtual knots can be defined. 
However, this product is not commutative in general in contrast to classical long knots again. 

In this paper, we introduce an example of biquandle and define an invariant of long virtual knots through it. 
As an application, we give a necessary condition for two long virtual knots commuting. 
The rest of this paper is organized as follows. 
In Section~\ref{section2}, we define a ring and give it a biquandle structure. 
In Section~\ref{section3}, we recall the definition of long virtual knots and their biquandle colorings. 
In Section~\ref{section4}, we define an invariant of long virtual knots using our biquandle coloring. 
In Section~\ref{section5}, using the invariant, we give a necessary condition for two long virtual knots commuting.


\section{Definition and biquandle structure of $X$}\label{section2}
In this section, we give a formal definition of our ring $X$ (in subsection~\ref{defX}) and give a biquandle structure on it (in subsection~\ref{biqdlX}). 
Before that, we state an informal definition of $X$ for our use: 
The ring $X$ is the polynomial ring with integer coefficients which has three non-commutative variables $s$, $u$ and $t$. 
For these variables, the following holds. $(1)$ $s$ is invertible, and $u$ and $t$ are non-invertible. $(2)$ elements of $X$ are formal power series about $u$ and $t$. (3) $us=su$ and $sts^{-1}=(1-u)t(1+t)^{-1}$.

\subsection{Definition of ring $X$}\label{defX}
Let $S=\{s, s^{-1}, u,t\}$ be a set containing four letters, and consider non-commutative free algebra $\mathbb{Z} \langle S \rangle$ on $S$. 
Give degrees to four variables  as ${\rm deg}(s)={\rm deg}(s^{-1})=0$ and ${\rm deg}(u)={\rm deg}(t)=1$. 
Take the degree completion and call the resultant $R=\mathbb{Z}\langle s, s^{-1}\rangle \langle \langle u,t\rangle \rangle $. 
For non-negative integer $k$, set $G^{k}$ to be the two sided ideal of $R$ generated by monomials whose degrees are grater than or equal to $k$. 
This gives a filtration $R=G^{0}\supset G^{1}\supset \dots$. 

\begin{defini}
Let $J$ be the closure of the two sided ideal of $R$ generated by $ss^{-1}-1$, $s^{-1}s-1$, $su-us$ and $sts^{-1}-(1-u)t(1+t)^{-1}$, where $(1-A)^{-1}$ denotes $1+A+A^{2}+\dots$ for $A\in G^{1}$. 
Then set $X$ to be the quotient $R/J$. 
We  use the same letters for the images of $s$, $s^{-1}$, $u$ and $t$. 
Moreover, we have a filtration $X=F^{0}\supset F^{1}\supset \dots$ by setting $F^{k}$ to be $(G^{k}+J)/J$. 
\end{defini}

\begin{rmk}
By relations on $X$ induced by $J$, we can deduce $s^{-1}ts=(1-u-t)^{-1}t$ as follows.
 \begin{align}
sts^{-1}=(1-u)t(1+t)^{-1} &\Longleftrightarrow sts^{-1}(1+t)s=(1-u)ts \notag \\
                                  &\Longleftrightarrow ts^{-1}(1+t)s=(1-u)s^{-1}ts \notag \\
                                  & \Longleftrightarrow t\left(1+s^{-1}ts\right)=(1-u)s^{-1}ts \notag \\
                                  & \Longleftrightarrow  t=(1-u-t)s^{-1}ts \notag \\
                                  & \Longleftrightarrow  (1-u-t)^{-1}t=s^{-1}ts \notag
 \end{align}

We list some relations frequently used: 
 \begin{align}
             s(u+t)s^{-1}&=u+sts^{-1}=u+(1-u)t(1+t)^{-1} \notag \\
                                  &=\left(u(1+t)+(1-u)t\right)(1+t)^{-1}=(u+t)(1+t)^{-1} \notag \\
         \notag \\
             s(1-u-t)s^{-1}&=1-s(u+t)s^{-1}=1-(u+t)(1+t)^{-1} \notag \\
                                &=\left( (1+t)-(u+t)\right)(1+t)^{-1} =(1-u)(1+t)^{-1} \notag 
 \end{align}

 \begin{align}
             s^{-1}(u+t)s&= u+s^{-1}ts=u+(1-u-t)^{-1}t =(1-u-t)^{-1}\left((1-u-t)u+t\right) \notag \\
                            &=(1-u-t)^{-1}(u+t-u^{2}-tu)=(1-u-t)^{-1}(u+t)(1-u) \notag \\
\notag \\
             s^{-1}(1-u-t)s&=1-s^{-1}(u+t)s=1-(1-u-t)^{-1}(u+t)(1-u) \notag \\
                                &=(1-u-t)^{-1} \left( (1-u-t)-(u+t)(1-u)\right) \notag \\
                                &=(1-u-t)^{-1}(1-2u-2t+u^{2}+tu) \notag
 \end{align}
\end{rmk}

Next, we give a``basis" of $X$ as $\mathbb{Z}-$module. 
\begin{defini}\label{def_base}
For a non-negative integer $k$, let $W_{k}(u,t)$ denote the set of sequences of letters of length $k$ consisting of $u$ and $t$, where the sequence of letters of length $0$ is the empty word. 
For example, $W_{2}(u,t)=\{ u^{2}, ut, tu, t^{2}\}$. 
And let $W(u,t)$ denote $\bigcup \limits_{k=0}^{\infty} W_{k}(u,t)$. 
\end{defini}

\begin{prop}\label{prop_base}
Every element of $X$ is represented by an infinite $\mathbb{Z}-$linear combination of $\mathcal{B}=\{ s^{n}w | n\in \mathbb{Z}, w\in W(u,t)\}$ uniquely. 
\end{prop}
\begin{proof}
By $us=su$, $sts^{-1}=(1-u)t(1+t)^{-1}$ and $s^{-1}ts=(1-u-t)^{-1}t$, we can move variables $s^{\pm1}$ to the left i.e. every element of $X$ can be represented as an infinite linear combination of $\mathcal{B}$. 

For the uniqueness, it is suffice to show that $x=\sum \limits_{b\in \mathcal{B}} n_{b}b =0$, where $n_{b} \in \mathbb{Z}$ implies each $n_{b}$ is zero. 
For contradiction, suppose that there is non-zero $n_{b}$. 
Let $l$ be the integer ${\rm min} \{ {\rm deg}(b)| b \in \mathcal{B},  n_{b}\neq0\}$. 
We claim that $F^{k}/F^{k+1}$ is isomorphic to $\mathbb{Z}[s^{\pm1}]\otimes\mathbb{Z}W_{k}(u,t)$ as $\mathbb{Z}-$module for any non-negativve integer $k$. 
By the second homomorphism theorem, we see that $F^{k}=\left(G^{k}+J\right)/J \cong G^{k}/\left(G^{k}\cap J\right)$. 
Under this isomorphism, the submodule $F^{k+1}=\left(G^{k+1}+J\right)/J$ of $F^{k}=\left(G^{k}+J\right)/J $ is mapped to $\left(G^{k+1}+G^{k}\cap J\right)/\left(G^{k}\cap J\right)$. 
Thus we see that $F^{k}/F^{k+1}$ is isomorphic to $G^{k}/\left(G^{k+1} + G^{k}\cap J\right)$. 
Note that considering $G^{k}$ modulo $G^{k+1}$ means that we ignore parts whose degrees are grater than $k$. 
Under this ignoring, relations induced by $G^{k}\cap J$ are generated by $ss^{-1}=1$, $s^{-1}s=1$, $us=su$ and $sts^{-1}=t$. 
Thus we can move $s^{\pm1}$ to the left and there are no relations between $u$ and $t$. 
Hence $F^{k}/F^{k+1}$ is isomorphic to $\mathbb{Z}[s^{\pm1}]\otimes\mathbb{Z}W_{k}(u,t)$, which has basis $\mathcal{B}_k=\{ s^{n}w | n\in \mathbb{Z}, w\in W_{k}(u,t)\}$ as $\mathbb{Z}-$module. 
Then by projecting $x=\sum \limits_{b\in \mathcal{B}} n_{b}b =0$ modulo $F^{l+1}$, we get a contradiction. 
\end{proof}

We determine the unit elements of $X$.
\begin{lem}\label{unit}
An element $x\in X$ is a unit element if and only if the degree $0$ part of $x$ is the form of $\pm s^{n}$ for some integer $n$. 
\end{lem}

\begin{proof}
Suppose $x\in X$ is a unit element. Then by projecting to $X/F^{1}=\mathbb{Z}[s^{\pm1}]$, we see that the degree $0$ part of $x$ is the form of $\pm s^{n}$ for some integer $n$. 
Conversely, suppose $x$ is the form of $\epsilon s^{n}+A$, where $\epsilon \in \{\pm1\}$, $n\in \mathbb{Z}$ and $A\in F^{1}$. 
Then $\epsilon \left(1+\epsilon s^{-n}A\right)^{-1}s^{-n}$ is the inverse of $x$. 
\end{proof}

\subsection{Biquandle structure of $X$}\label{biqdlX}
First, we recall the definition of biquandles. 
We adopt more memorable definition of biquandles than the original one, the author thinks. 
See \cite{elhamdadi} for example. 
\begin{defini}\label{def_biqdl}(Biquandle)\\
A {\it biquandle} $\left( B; \underline{\triangleright}, \overline{\triangleright}\right)$ is a triad of a set $B$ and operations $\underline{\triangleright},\ \overline{\triangleright}: B\times B \rightarrow B$ which satisfies: 
\begin{itemize}
\item[(i)] for all $x \in B$, $x\ \underline{\triangleright}\ x = x\ \overline{\triangleright}\ x$, 
\item[(ii)] for all $y \in B$, the maps $(\cdot \ \underline{\triangleright} \ y),\ (\cdot \  \overline{\triangleright} \ y) : B\rightarrow B$ are invertible, 
\item[(iii)] the map $S:B\times B\rightarrow B\times B$ defined by $S\left(x,y \right)=\left( y\ \overline{\triangleright}\ x, x\ \underline{\triangleright}\ y\right)$ is invertible, and
\item[(iv)] for all $x,y,z\in B$, the following exchange lows hold. 
  \begin{align}
(x\ \underline{\triangleright}\ y)\ \underline{\triangleright} \ (z\ \underline{\triangleright}\ y)  &= (x\ \underline{\triangleright}\ z)\ \underline{\triangleright}\ (y\ \overline{\triangleright}\ z) \notag \\
(x\ \underline{\triangleright}\ y)\ \overline{\triangleright} \ (z\ \underline{\triangleright}\ y)  &= (x\ \overline{\triangleright}\ z)\ \underline{\triangleright}\ (y\ \overline{\triangleright}\ z) \notag  \\
(x\ \overline{\triangleright}\ y)\ \overline{\triangleright} \ (z\ \overline{\triangleright}\ y)  &= (x\ \overline{\triangleright}\ z)\ \overline{\triangleright}\ (y\ \underline{\triangleright}\ z) \notag 
     \end{align}
\end{itemize}
\end{defini}

\begin{example}(Alexander biquandle)\\
The triad $\left( \mathbb{Z}[U^{\pm1}, V^{\pm1}]; \underline{\triangleright}, \overline{\triangleright}\right)$ defined by $x\ \underline{\triangleright}\ y=U^{-1}x+\left(V-U^{-1}\right)y$ and $x\ \overline{\triangleright} \ y=Vx$ is a biquandle. 
This is called {\it the Alexander biquandle}. 
The natations here may differ from the ordinary definition. 
\end{example}

We give a biquandle structure to $X$. 
\begin{thm}
$\left( X; \underline{\triangleright}, \overline{\triangleright}\right)$ is a biquandle by defining $x\ \underline{\triangleright}\ y=s^{-1}\cdot x-s^{-1}(u+t)\cdot y$ and $x\ \overline{\triangleright}\ y=s^{-1}(1-u)\cdot x-s^{-1}t\cdot y$ for $x,y \in X$. 
\end{thm}

\begin{proof}
We check the conditions for biquandles in Definition~\ref{def_biqdl}. 
\begin{itemize}
\item[(i)] $x\ \underline{\triangleright}\ x=s^{-1}(1-u-t)x=x\ \overline{\triangleright}\ x$ 

\item[(ii)] 
\begin{align}
&z=x\ \underline{\triangleright}\ y=s^{-1}\cdot x -s^{-1}(u+t)\cdot y \Longleftrightarrow x=s\cdot z+(u+t)\cdot y \notag \\
&z=x\ \overline{\triangleright}\ y=s^{-1}(1-u)\cdot x-s^{-1}t\cdot y \Longleftrightarrow x=(1-u)^{-1}s\cdot z+(1-u)^{-1}t\cdot y \notag
\end{align}

\item[(iii)]
\begin{align}
&\begin{cases}
z=y\ \overline{\triangleright}\ x=s^{-1}(1-u)\cdot y-s^{-1}t\cdot x\\
w=x\ \underline{\triangleright}\ y=s^{-1}\cdot x-s^{-1}(u+t)\cdot y
\end{cases} \notag \\ 
\Longleftrightarrow &
\begin{cases}
s\cdot z= (1-u)\cdot y-t\cdot x\\
x=(u+t)\cdot y+s\cdot w
\end{cases}  \notag \\ 
\Longleftrightarrow &
\begin{cases}
s\cdot z= (1-u)\cdot y-t \left((u+t)\cdot y+s\cdot w\right)\\
x=(u+t)\cdot y+s\cdot w
\end{cases} \notag \\ 
\Longleftrightarrow &
\begin{cases}
y=\left(1-u-tu-t^{2}\right)^{-1} s\cdot z+\left(1-u-tu-t^{2}\right)^{-1} ts\cdot w\\
x=(u+t)\cdot y+s\cdot w
\end{cases} \notag \\ 
\Longleftrightarrow &
\begin{cases}
y=\left(1-u-tu-t^{2}\right)^{-1} s\cdot z+\left(1-u-tu-t^{2}\right)^{-1} ts\cdot w\\
x=(u+t)\left(\left(1-u-tu-t^{2}\right)^{-1} s\cdot z+\left(1-u-tu-t^{2}\right)^{-1} ts\cdot w\right)+s\cdot w
\end{cases} \notag \\ 
\Longleftrightarrow &
\begin{cases}
x=(u+t)\left(1-u-tu-t^{2}\right)^{-1} s\cdot z+\left(s+(u+t)\left(1-u-tu-t^{2}\right)^{-1} ts)\right)\cdot w \\
y=\left(1-u-tu-t^{2}\right)^{-1} s\cdot z+\left(1-u-tu-t^{2}\right)^{-1} ts\cdot w
\end{cases} \notag 
\end{align}

\item[(iv)]
 \begin{itemize}
  \item[]
\begin{align}
&(x\ \underline{\triangleright}\ y)\ \underline{\triangleright}\ (z\ \underline{\triangleright}\ y) \notag \\ 
=& s^{-1}\left(s^{-1}\cdot x-s^{-1}(u+t)\cdot y\right)-s^{-1}(u+t)\left(s^{-1}\cdot z-s^{-1}(u+t)\cdot y\right) \notag \\
=& s^{-2}\cdot x+\left(\left(s^{-1}(u+t)\right)^{2}-s^{-2}(u+t)\right)\cdot y-s^{-1}(u+t)s^{-1}\cdot z \notag 
 \end{align}
  \item[]
\begin{align}
&(x\ \underline{\triangleright}\ z)\ \underline{\triangleright}\ (y\ \overline{\triangleright}\ z) \notag \\ 
=& s^{-1}\left(s^{-1}\cdot x-s^{-1}(u+t)\cdot z\right)-s^{-1}(u+t)\left(s^{-1}(1-u)\cdot y-s^{-1}t\cdot z\right) \notag \\
=& s^{-2}\cdot x-s^{-1}(u+t)s^{-1}(1-u)\cdot y+\left(s^{-1}(u+t)s^{-1}t-s^{-2}(u+t)\right)\cdot z \notag 
 \end{align}
\item[]
\item[] Note that the following two equations hold:
\begin{align}
&\left(\left(s^{-1}(u+t)\right)^{2}-s^{-2}(u+t)\right)-\left(-s^{-1}(u+t)s^{-1}(1-u)\right) \notag  \\
=&s^{-2}\left( s(u+t)s^{-1}(u+t)-(u+t)+s(u+t)s^{-1}(1-u)\right) \notag \\
=&s^{-2}\left(s(u+t)s^{-1}(1+t)-(u+t)\right) \notag \\
=&s^{-2}\left( (u+t)(1+t)^{-1}\cdot (1+t)-(u+t)\right) \notag \\
=&0 \notag
\end{align} 

\begin{align}
&\left(s^{-1}(u+t)s^{-1}t-s^{-2}(u+t)\right)-\left(-s^{-1}(u+t)s^{-1}\right) \notag  \\
=& s^{-2}\left(s(u+t)s^{-1}t-(u+t)+s(u+t)s^{-1}\right) \notag \\
=& s^{-2}\left(s(u+t)s^{-1}(1+t)-(u+t)\right) \notag \\
=& s^{-2}\left( (u+t)(1+t)^{-1} \cdot(1+t)-(u+t)\right) \notag \\
=&0 \notag
\end{align}

  \item[]
\begin{align}
&(x\ \underline{\triangleright}\ y)\ \overline{\triangleright}\ (z\ \underline{\triangleright}\ y) \notag \\ 
=& s^{-1}(1-u)\left(s^{-1}\cdot x-s^{-1}(u+t)\cdot y\right)-s^{-1}t\left(s^{-1}\cdot z-s^{-1}(u+t)\cdot y\right) \notag \\
=& s^{-2}(1-u)\cdot x+\left(s^{-1}ts^{-1}(u+t)-s^{-2}(1-u)(u+t)\right)\cdot y-s^{-1}ts^{-1}\cdot z \notag 
 \end{align}
  \item[]
\begin{align}
&(x\ \overline{\triangleright}\ z)\ \underline{\triangleright}\ (y\ \overline{\triangleright}\ z) \notag \\ 
=& s^{-1}\left(s^{-1}(1-u)\cdot x-s^{-1}t\cdot z\right)-s^{-1}(u+t)\left(s^{-1}(1-u)\cdot y-s^{-1}t\cdot z\right) \notag \\
=& s^{-2}(1-u)\cdot x-s^{-1}(u+t)s^{-1}(1-u)\cdot y+\left(s^{-1}(u+t)s^{-1}t-s^{-2}t\right)\cdot z \notag 
 \end{align}
\item[]
\item[] Note that the following two equations hold:
\begin{align}
&\left(s^{-1}ts^{-1}(u+t)-s^{-2}(1-u)(u+t)\right)-\left(-s^{-1}(u+t)s^{-1}(1-u)\right)  \notag  \\
=& s^{-2}\left( sts^{-1}(u+t)-(1-u)(u+t)+s(u+t)s^{-1}(1-u)\right) \notag \\
=& s^{-2}\left( sts^{-1}(u+t)-(1-u)(u+t)+u(1-u)+sts^{-1}(1-u)\right) \notag \\
=& s^{-2}\left( sts^{-1}(1+t)-(1-u)(u+t)+(1-u)u\right) \notag \\
=& s^{-2}\left( (1-u)t(1+t)^{-1}\cdot(1+t)-(1-u)t\right) \notag \\
=&0 \notag
\end{align} 

\begin{align}
& \left(s^{-1}(u+t)s^{-1}t-s^{-2}t\right)-\left(-s^{-1}ts^{-1}\right) \notag  \\
=&s^{-2}\left(s(u+t)s^{-1}t-t+sts^{-1}\right) \notag \\
=&s^{-2}\left( (u+t)(1+t)^{-1}\cdot t-t+(1-u)t(1+t)^{-1}\right)  \notag \\
=& s^{-2}\left( (u+t)-(1+t)+(1-u)\right)t(1+t)^{-1} \notag \\
=&0 \notag
\end{align}

  \item[]
\begin{align}
&(x\ \overline{\triangleright}\ y)\ \overline{\triangleright}\ (z\ \overline{\triangleright}\ y) \notag \\ 
=& s^{-1}(1-u)\left(s^{-1}(1-u)\cdot x-s^{-1}t\cdot y\right)-s^{-1}t\left(s^{-1}(1-u)\cdot z-s^{-1}t\cdot y\right) \notag \\
=& s^{-2}(1-u)^{2}\cdot x+\left((s^{-1}t)^{2}-s^{-2}(1-u)t\right)\cdot y-s^{-1}ts^{-1}(1-u)\cdot z \notag 
 \end{align}
  \item[]
\begin{align}
&(x\ \overline{\triangleright}\ z)\ \overline{\triangleright}\ (y\ \underline{\triangleright}\ z) \notag \\ 
=& s^{-1}(1-u)\left(s^{-1}(1-u)\cdot x-s^{-1}t\cdot z\right)-s^{-1}t\left(s^{-1}\cdot y-s^{-1}(u+t)\cdot z\right) \notag \\
=& s^{-2}(1-u)^{2}\cdot x-s^{-1}ts^{-1}\cdot y+\left(s^{-1}ts^{-1}(u+t)-s^{-2}(1-u)t\right)\cdot z \notag 
 \end{align}
\item[]
\item[] Note that the following two equations hold:
\begin{align}
& \left((s^{-1}t)^{2}-s^{-2}(1-u)t\right)- \left(-s^{-1}ts^{-1} \right) \notag  \\
=&s^{-2}\left( sts^{-1}t-(1-u)t+sts^{-1}\right)  \notag \\
=& s^{-2}\left( sts^{-1}(1+t)-(1-u)t\right) \notag \\
=& s^{-2}\left( (1-u)t(1+t)^{-1}\cdot(1+t)-(1-u)t\right) \notag \\
=&0 \notag
\end{align} 

\begin{align}
& \left(s^{-1}ts^{-1}(u+t)-s^{-2}(1-u)t\right)- \left(-s^{-1}ts^{-1}(1-u)\right) \notag  \\
=&s^{-2}\left( sts^{-1}(u+t)-(1-u)t+sts^{-1}(1-u)\right) \notag \\
=& s^{-2}\left(sts^{-1}(1+t)-(1-u)t \right) \notag \\
=& s^{-2}\left( (1-u)t(1+t)^{-1}\cdot(1+t)-(1-u)t\right) \notag \\
=&0 \notag
\end{align} 

 \end{itemize}

\end{itemize}
\end{proof}

\section{Long virtual knots and biquandle colorings}\label{section3}
In this section, first we recall the definition of long virtual knots and some terminologies (in Subsection~\ref{subsec_long virtual knots}). 
Then we recall biquandle colorings of long virtual knots (in Subsection~\ref{subsec_biquandle coloring}). 

\subsection{Long virtual knots}\label{subsec_long virtual knots}

\begin{defini}
{\it A long virtual knot diagram} is a smooth immersion $\phi : \mathbb{R} \longrightarrow \mathbb{R}^{2}$ such that 
\begin{itemize}
\item[(1)] there is a positive real number $r$ such that $\phi(x)=(x,0)$ for any real number $x$ with $|x|>r$, 
\item[(2)] each intersection point is double and transverse, and
\item[(3)] each intersection point is endowed with classical (with a choice for overpass and underpass specified) or virtual crossing  (which is depicted as encircled double point) structure. 
\end{itemize}
\end{defini}

\begin{defini}
{\it A long virtual knot} is an equivalence class of long virtual knot diagrams under generalized Reidemeister moves as shown in Figure~\ref{generalized_reidemeister}. 
Each long virtual knot are oriented by the standard orientation of $\mathbb{R}$. 
With the orientation, each classical crossing is classified into two classes, {\it positive crossings} and {\it negative crossings} as in Figure~\ref{sign_crossing}. 
For each classical crossing $c$, let $\epsilon (c)$ denote the number $+1$ or $-1$ if $c$ is positive crossing or negative crossing, respectively. 
A long virtual knot represented by a long virtual knot diagram without virtual crossings is called {\it classical}. 
\end{defini}

\begin{figure}[htbp]
 \begin{center}
  \includegraphics[width=140mm]{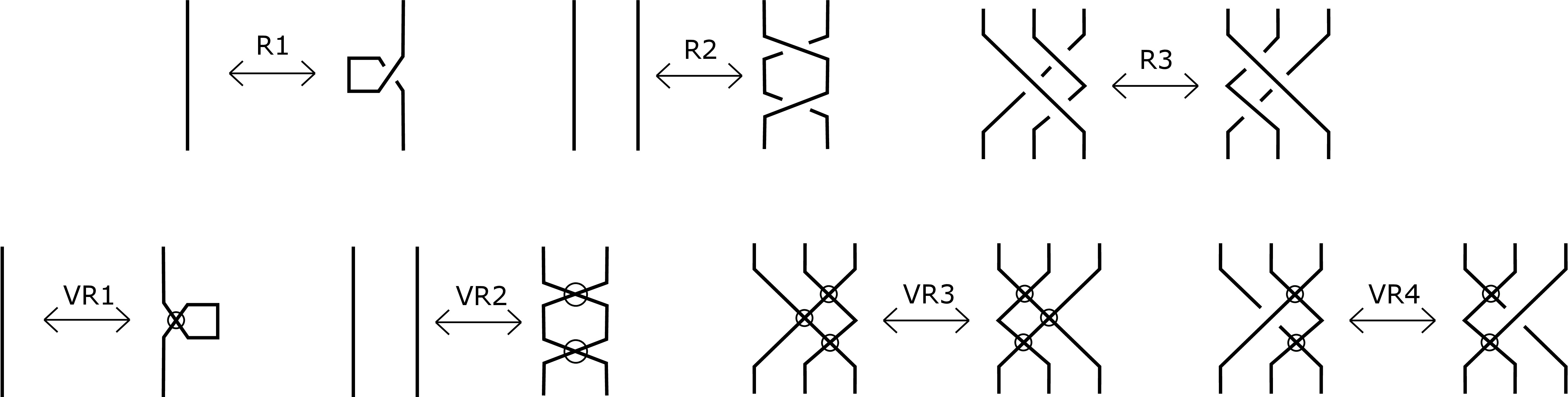}
 \end{center}
 \caption{Generalized Reidemeister moves}
 \label{generalized_reidemeister}
\end{figure}

\begin{figure}[htbp]
 \begin{center}
  \includegraphics[width=50mm]{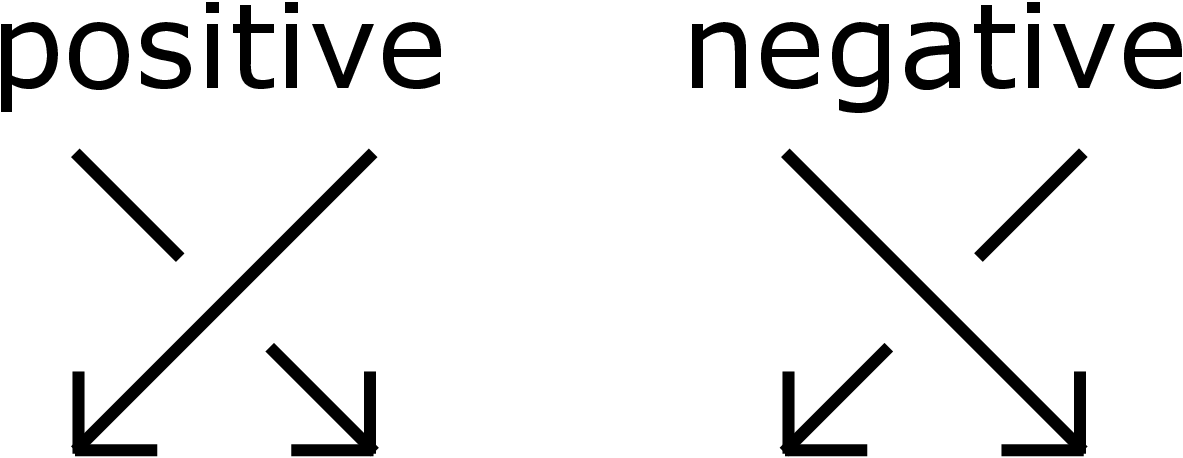}
 \end{center}
 \caption{The sign of classical crossings}
 \label{sign_crossing}
\end{figure}

The set of long virtual knots admits a structure of monoid. 
\begin{defini}
Let $K_{1}$ and $K_{2}$ be two long virtual knots. 
Take long virtual knot diagrams $D_{1}$ and $D_{2}$ representing $K_{1}$ and $K_{2}$, respectively. 
Let $D_{1}\cdot D_{2}$ denote the diagram obtained by concatenating $D_1$ and $D_2$ in this order. 
Then the long virtual knot represented by $D_{1}\cdot D_{2}$ is called {\it the product of} $D_{1}$ {\it and} $D_2$, and denoted by $K_{1}\cdot K_{2}$. 
\end{defini}

There ia a useful description for long virtual knots called the pointed Gauss diagram. 
Many invariants of long virtual knots can be computed relatively easily through it. 

\begin{defini}(Pointed Gauss diagram)\\
Take a counterclockwise oriented circle $S^{1}$ with one specified point $p$ on it. 
For a long virtual knot diagram $D$, fix an orientation preserving immersion $\phi$ from $S^{1}\setminus \{p\}$ to the image of $D$. 
For each classical crossing $c$ of $D$, the preimage $\phi^{-1}(c)$ of $c$ consists of two points of $S^{1}\setminus \{p\}$. 
Let $c_{O}$ and $c_{U}$ be the points of $\phi^{-1}(c)$ corresponding to for the points being on  the overarc and the under arc, respectively. 
Connect $c_{O}$ and $c_{U}$ by a dashed arrow oriented from $c_{O}$ to $c_{U}$ for each classical crossing of $D$. 
Assign this dashed arrow the symbol $+$ or $-$ according to the sign of the crossing $c$. 
The obtained pointed circle with dashed arrows is called {\it the pointed Gauss diagram of} $D$, denoted by $G(D)$. 
An example of a pointed Gauss diagram of a long virtual knot is presented in Figure~\ref{gaussdiagram_k3(4)}. 
\end{defini}

\begin{figure}[htbp]
 \begin{center}
  \includegraphics[width=60mm]{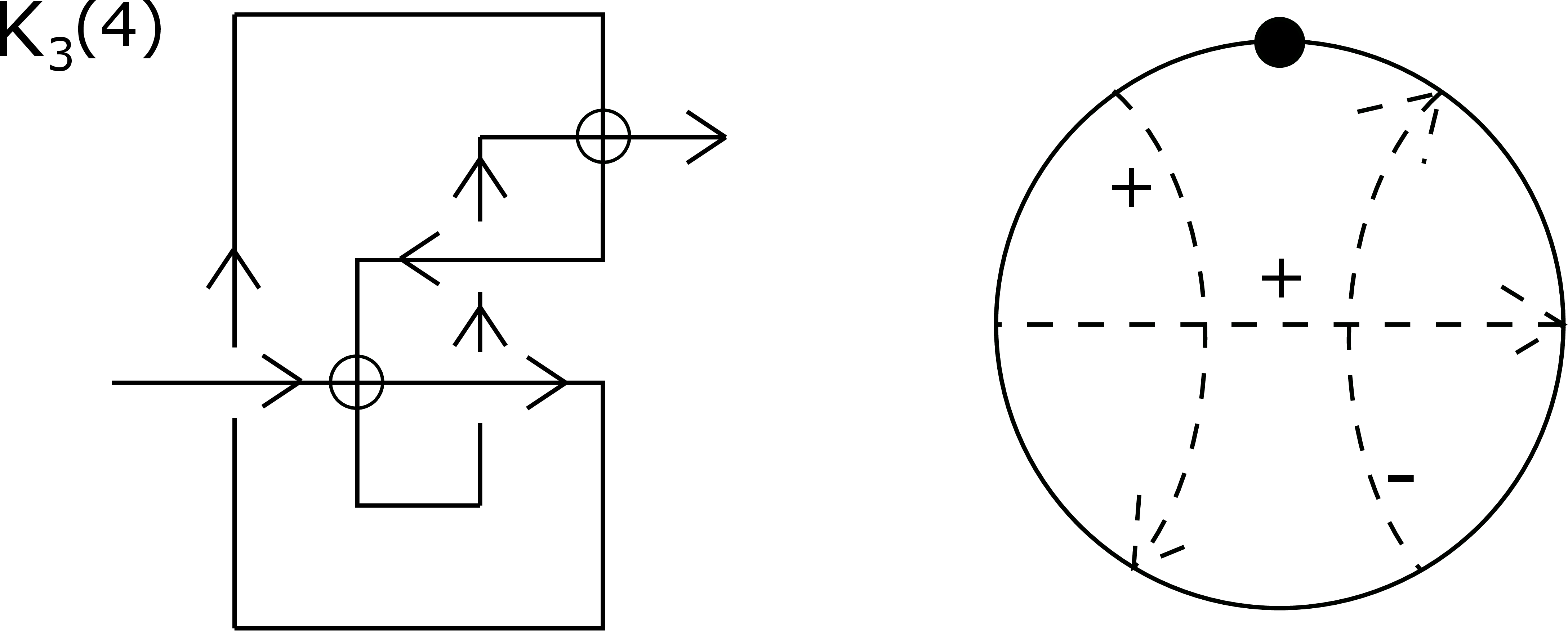}
 \end{center}
 \caption{An example of a pointed Gauss diagram}
 \label{gaussdiagram_k3(4)}
\end{figure}

We will introduce a well-known invariant of long virtual knots. 
Before that, we prepare the terminologies. 

\begin{defini}(EO and EU \cite{manturov})\\
Let $D$ be a long virtual knot diagram. 
For each classical crossing $c$ of $D$, either one of  ``we reach $c$ along overarc first" or ``we reach $c$ along underarc first" holds. 
The set of crossings satisfying the former is denoted by EO, {\it early overarc}, and that of crossings satisfying the latter is denoted by EU, {\it earky underarc}.
\end{defini}

\begin{defini}(Index of a crossing)\\
Let $G(D)$ be the pointed Gauss diagram of a long virtual knot diagram $D$. 
For each dashed arrow in $G(D)$, we assign symbols $+$ and $-$ to the endpoints of the dashed arrow so that the symbol assigned to the tail is the same as that of the dashed arrow and that the symbol assigned to the head is the opposite to that of the dashed arrow. 
For each classical crossing $c$ of $D$, corresponding dashed arrow in $G(D)$ cuts unique interval without the specified point $p$ from the circle. 
The signed sum of symbols on the interior of the interval is denoted by $\alpha(c)$ and called {\it the index of }$c$. 
\end{defini}

\begin{fact}\label{writhe polynomial}
Let $K$ be a long virtual knot. 
Take a long virtual knot diagram $D$. 
Then  
\begin{align}
 \sum \limits_{c\in {\rm EO}} \epsilon(c)\cdot(s^{\alpha(c)}-1)  +  \sum \limits_{d\in {\rm EU}} \epsilon(d)\cdot(s^{-\alpha(d)}-1) \in \mathbb{Z}[s^{\pm1}] \notag
\end{align}
is independent of the choice of $D$. 
Thus this becomes an invariant of long virtual knots, denoted by $W_{K}(s)$. 
\end{fact}

\begin{rmk}
The above $W_{K}(s)$ is the invariant called the writhe polynomial of the virtual knot $\hat{K}$ obtained by closing $K$. 
For example, $W_{K_{3}(4)}(s)=(s-1)+(s^{-2}-1)-(s^{-1}-1)=s-1-s^{-1}+s^{-2}$ for the long virtual knot $K_{3}(4)$ in Figure~\ref{gaussdiagram_k3(4)}. 
Note that $W_{K}(s)$ is different from so-called the writhe polynomial of long virtual knots defined through long virtual knot diagrams by :
\begin{align}
 \sum \limits_{c\in {\rm EO}\cup {\rm EU}} \epsilon(c)\cdot(s^{\alpha(c)}-1)  \in \mathbb{Z}[s^{\pm1}] \notag
\end{align}
\end{rmk}

A projection from the set of long virtual knots to some subset can be defined as follows. 
\begin{defini}\label{eonization}
Let $\mathcal{EO}$ be a set of long virtual knots each of which admits a long virtual knot diagram without EU crossings. 
For a long virtual knot diagram $D$, set ${\rm EO}(D)$ to be a long virtual knot diagram obtained by performing crossing changes at each EU crossing. 
It can be checked that ${\rm EO}(D)$ and ${\rm EO}(D')$ are related by generalized Reidemeister moves if two long virtual knot diagrams $D$ and $D'$ are related by generalized Reidemeister moves. 
Thus we can define ${\rm EO}(K) \in \mathcal{EO}$ for a long virtual knot $K$ as the long virtual knot represented by ${\rm EO}(D)$, where $D$ is a long virtual knot diagram representing $K$. 
This ${\rm EO}$ is a projection from the set of long virtual knots to $\mathcal{EO}$. 
Moreover, this ${\rm EO}$ preserves the products. In other words, ${\rm EO}(K_{1} \cdot K_{2})={\rm EO}(K_{1})\cdot {\rm EO}(K_{2})$ holds for long virtual knots $K_{1}$ and $K_{2}$. 
\end{defini}

\begin{lem}\label{eo_writhe}
Let $K$ be a long virtual knot. Take a long virtual knot diagram $D$ representing $K$. 
Then $W_{{\rm EO}(K)}(s)$ can be computed through $D$ as:
 \begin{align}
 W_{{\rm EO}(K)}(s)=\sum \limits_{c\in {\rm EO}} \epsilon(c)\cdot(s^{\alpha(c)}-1)  -  \sum \limits_{d\in {\rm EU}} \epsilon(d)\cdot(s^{\alpha(d)}-1)  \notag
 \end{align}
\end{lem}

\begin{proof}
Note that the sign of a crossing is reversed if the crossing change is operated at this crossing. 
Note also that the signs on the circle do not change under crossing changes. 
Thus the index of the crossing of ${\rm EO}(D)$ is the same as that of the corresponding crossing of $D$. 
Thus the contribution of each EU crossing $d$ of $D$ to $W_{{\rm EO}(K)}(s)$ is $-\epsilon (d)\cdot (s^{\alpha(d)}-1)$. 
\end{proof}

\begin{rmk}
It can be shown that ${\rm EO}(K)={\rm EO}(K')$ holds for two long virtual knots $K$ and $K'$ if and only if $K$ and $K'$ are related by crossing changes. 
Thus for an invariant $I(\cdot)$ of long virtual knots, the invariant $I\circ {\rm EO}(\cdot)$ is an invariant of long virtual knots modulo crossing changes.
\end{rmk}

There is a concept of sieblings for long virtual knots. 
\begin{defini}(Sieblings)\\
For a long virtual knot $K$, take its long virtual knot diagram $D$. 
Let $-D$ denote the long virtual knot diagram obtained by reversing the orientation of $D$. 
Let $D^{\#}$ denote the long virtual knot diagram obtained by replacing each classical crossing of $D$ by a crossing of the opposite sign. 
Let $D^{*}$ denote the long virtual knot diagram obtained by applying a planar reflection of $\mathbb{R}^{2}$ to $D$. 
Set $-K$, $K^{\#}$ and $K^{*}$  to be long virtual knots represented by $-D$, $D^{\#}$ and $D^{*}$, respectively, and call them {\it the reverse of} $K$, {\it the vertical mirror} of $K$ and {\it the horizontal mirror} of $K$, respectively. 
\end{defini}

In \cite{yoshida}, Yoshida classified completely long virtual knots each of which admit long virtual knot diagram whose number of classical crossings are at most three as follows, where the trivial long virtual knot is a long virtual knot which is represented by a long virtual knot diagram with no classical and virtual crossings. 

\begin{fact}
Suppose that a long virtual knot $K$ is not the trivial long virtual knot $K_{0}(1)$ and that $K$ admits a long virtual knot diagram whose number of classical crossings is at most $3$. 
Then $K$ is equivalent to exactly one of the long virtual knots in Figures~\ref{computation_symmetry1}, \ref{computation_symmetry2}, \ref{computation_symmetry3}, \ref{computation_symmetry4}, \ref{computation_symmetry5}, \ref{computation_symmetry6}, \ref{computation_symmetry7}, \ref{computation_symmetry8}, \ref{computation_symmetry9}, \ref{computation_symmetry10}. 
\end{fact}

\subsection{Biquandle colorings of long virtual knot diagrams}\label{subsec_biquandle coloring}

\begin{defini}
For a biquandle $(B; \underline{\triangleright}, \overline{\triangleright})$ and a long virtual knot diagram $D$, a $(B; \underline{\triangleright}, \overline{\triangleright})-${\it coloring of } $D$  is an assignment to semiarcs (segments between classical crossings and $\pm \infty$) of $D$ with elements of $B$ satisfying the condition as each crassical crossing as in Figure~\ref{biquandle_coloring_crossing}. 
\end{defini}

\begin{figure}[htbp]
 \begin{center}
  \includegraphics[width=70mm]{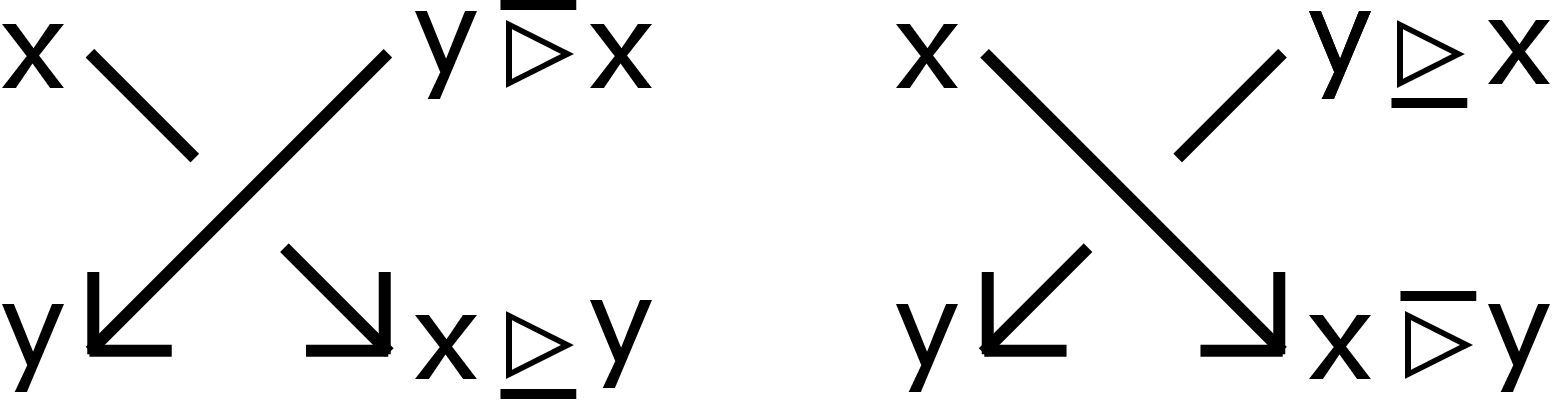}
 \end{center}
 \caption{Conditions of coloring at crossings}
 \label{biquandle_coloring_crossing}
\end{figure}

The following fact is well-known. See \cite{elhamdadi} for example. 
\begin{fact}\label{fact_inv}
Let $(B; \underline{\triangleright}, \overline{\triangleright})$ be a biquandle. 
Suppose that two long virtual knot diagrams $D$ and $D'$ are related by a finite sequence of generalized Reidemeister moves. 
Then there is the canonical bijection between the set of $(B; \underline{\triangleright}, \overline{\triangleright})-$colorings of $D$ and that of $D'$. 
Moreover, if $(B; \underline{\triangleright}, \overline{\triangleright})-$coloring $\mathcal{C}$ of $D$ is mapped to $(B; \underline{\triangleright}, \overline{\triangleright})-$coloring $\mathcal{C}'$ of $D'$ under this bijection, then the elements assigned to the initial arc (or the terminal arc)  of $D$ by $\mathcal{C}$ are the same as the elements assigned to the initial arc (or the terminal arc, respectiely)  of $D'$ by $\mathcal{C}'$. 
\end{fact}

\section{An invariant of long virtual knots defined through $(X; \underline{\triangleright}, \overline{\triangleright})-$colorings}\label{section4}

At first, we rewrite the conditions for $(X; \underline{\triangleright}, \overline{\triangleright})-$colorings as in Figure~\ref{Xcoloring} for the computation. 

\begin{figure}[htbp]
 \begin{center}
  \includegraphics[width=100mm]{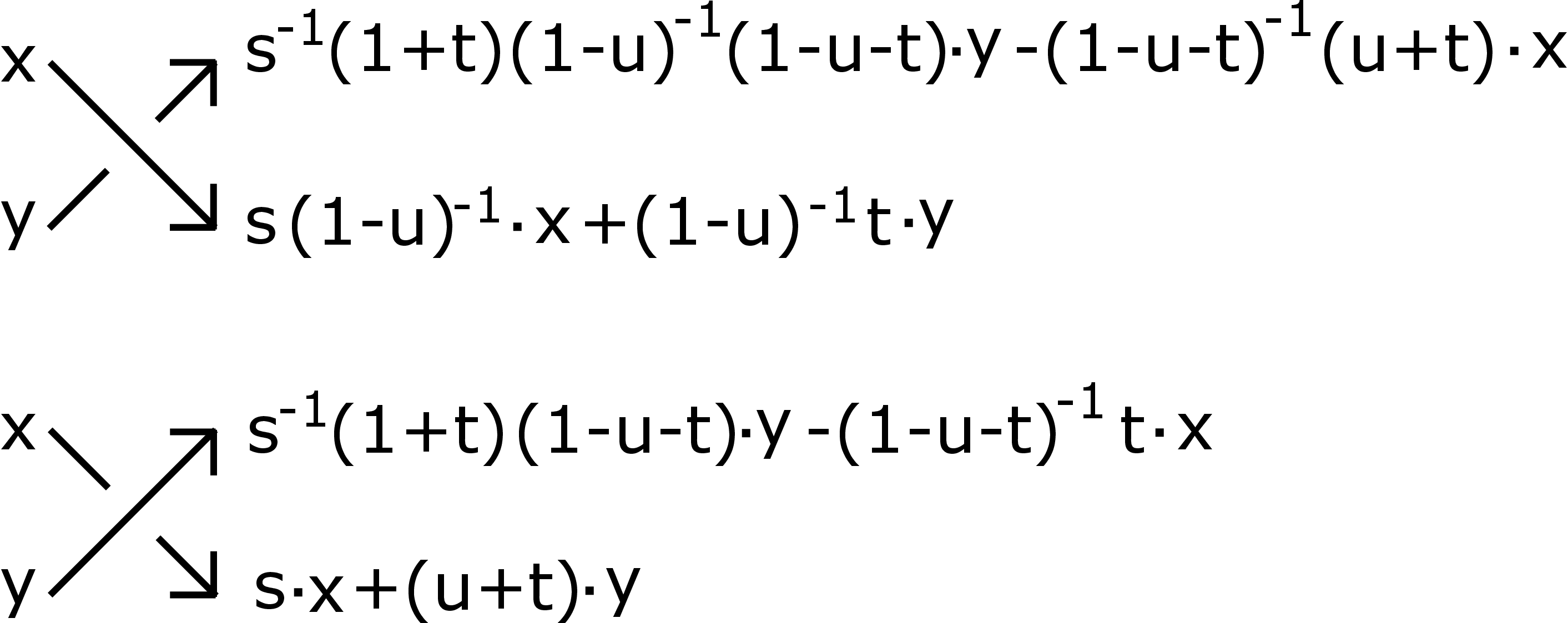}
 \end{center}
 \caption{$(X; \underline{\triangleright}, \overline{\triangleright})-$coloring at classical crossings for the computation}
 \label{Xcoloring}
\end{figure}

\begin{rmk}\label{substitute}
By substituting $t=0$ or $u=0$ on $X$, the colorings at classical crossings become as in Figure~\ref{Xcoloring_uort}. 
For the case $t=0$, $(X; \underline{\triangleright}, \overline{\triangleright})$ is equivalent to the Alexander biquandle by setting $U=s$ and $V=s^{-1}(1-u)$. 
Thus $(X; \underline{\triangleright}, \overline{\triangleright})-$coloring under $t=0$ can be regarded as an expansion of the Alexander biquandle coloring at $UV=1$. 
In \cite{mellor}, Mellor proved that the writhe polynomials of virtual knots are recovered from the generalized Alexander polynomials (which can be defined through the Alexander biquandle). 
To be precise, the first term of the generalized Alexander polynomial determines the writhe polynomial when the generalized Alexander polynomial is represented as the power series of $(1-UV)$. 
In \cite{mellor}, the higher writhe polynomials are also defined through the higher term of $(1-UV)$ in the generalized Alexander polynomials. 
For the case $u=0$, the coloring at each classical crossing is independent of the sign of the crossing. 
Thus $(X; \underline{\triangleright}, \overline{\triangleright})-$coloring under $u=0$ is a coloring for (long) flat virtual knots.

\begin{figure}[htbp]
 \begin{center}
  \includegraphics[width=80mm]{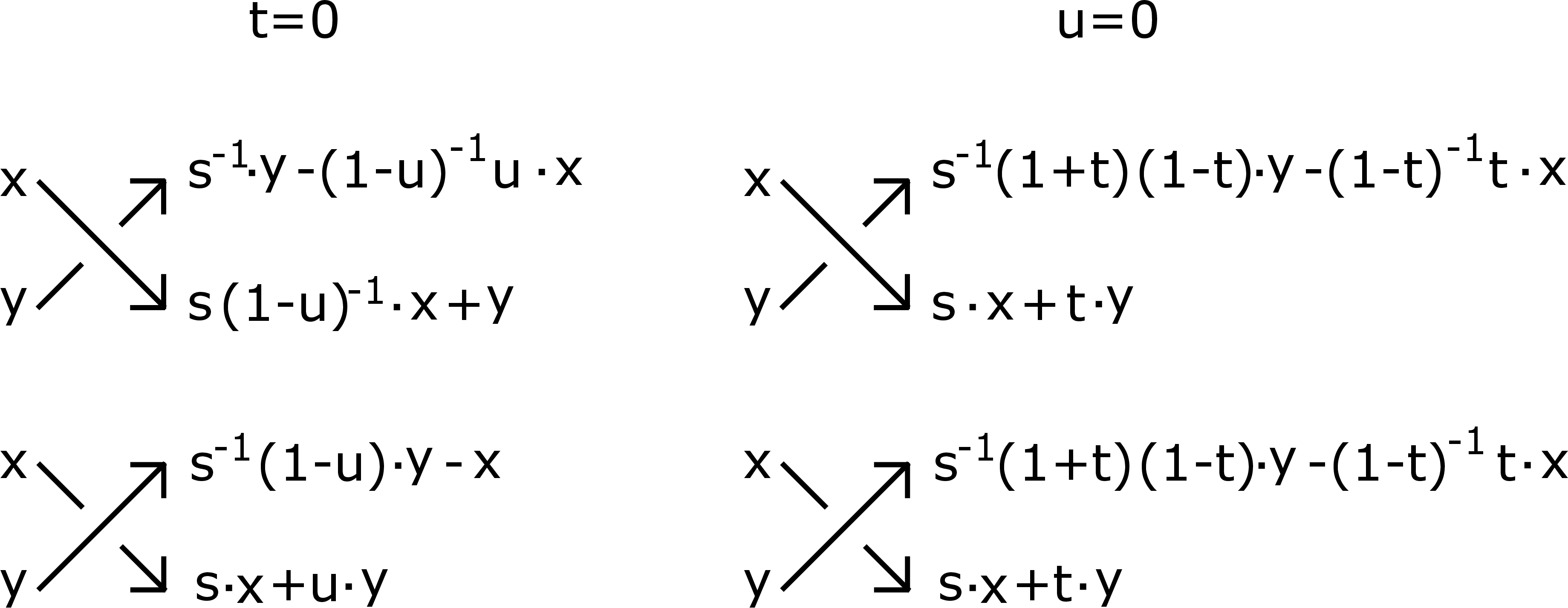}
 \end{center}
 \caption{$(X; \underline{\triangleright}, \overline{\triangleright})-$coloring at classical crossings \ \ Left: $t=0$ \ \ \ Right: $u=0$}
 \label{Xcoloring_uort}
\end{figure}

\end{rmk}

\subsection{Construction of our invariant}

The definition of our invariant relies on the following theorem. 
\begin{thm}\label{def_diagram}
Let $D$ be a long virtual knot diagram. 
Take arbitrary $a\in X$. 
Then there exists  unique $(X; \underline{\triangleright}, \overline{\triangleright})-$coloring $\mathcal{C}_{a}$ of $D$ such that the element assigned to the initial semiarc of $D$ is $a$. 
Moreover, there exists an element $F_{D}(s,u,t)\in X$, which is independent of $a$, such that the element assigned to the terminal semiarc of $D$ by $\mathcal{C}_{a}$ is $F_{D}(s,u,t)\cdot a$. 
\end{thm}

\begin{proof}
Let $n$ be the number of classical crossings of $D$. 
We give semiarcs of $D$ variables $x_{1}, x_{2},\dots, x_{2n+1}$ in order, from the initial semiarc to the terminal semiarc. 
We construct $2n+1$ equations as follows: 
The first equation is $x_{1}=a$. 
For $2\leq i\leq 2n+1$, the $i-$th equation is as in Figure~\ref{crossing_equation} depending on the types of classical crossings. 
In any cases of Figure~\ref{crossing_equation}, note that $x_{i}$ and $x_{j-1}$ can be same, that $x_{i}$ and $x_{i-1}$ cannot be same, that $x_{i}$ and $x_{i-1}$ cannot be $x_{i+1}$, and thus that the following holds:
\begin{itemize}
\item In the left hand side of the first equation, the coefficient of $x_1$ is $1$, and the coefficients of the other variables are $0$, and thus in $F^{1}$, and
\item in the left hand side of the $i-$th equation, the coefficient of $x_{i}$ is a unit element of $X$ (even if $x_{j-1}=x_{i}$ by Lemma~\ref{unit}), and the coefficients of the variables other than $x_{i-1}$ and $x_{i}$ are in $F^{1}$ for each $2\leq i\leq 2n+1$. 
\end{itemize} 
The set of $(X; \underline{\triangleright}, \overline{\triangleright})-$colorings of $D$ whose assignment to the initial semiarc are $a$ corresponds to the set of solutions of this simaltaneous $2n+1$ linear equations with coefficients $X$ with $2n+1$ valuables. 
We solve this by using the Gaussian elimination. See Example~\ref{example_computation}. 
We perform the following operation inductively: 
For $1\leq k\leq 2n$, suppose that 
\begin{itemize}
\item In the left hand side of the $k-$th equation, the coefficient of $x_k$ is $1$, and the coefficients of the other variables are in $F^{1}$, 
\item in the left hand side of the $i-$th equation, the coefficient of $x_{i}$ is a unit element of $X$, and the coefficients of the variables other than $x_{i-1}$ and $x_{i}$ are in $F^{1}$ for each $k+1\leq i\leq 2n+1$, and
\item in the $j-th$ equation for each $k\leq j\leq 2n+1$, there are no $x_{l}$ for $l<k$. 
\end{itemize} 
Then eliminate $x_{k}$ from the $i$-th equations for $k+1\leq i\leq 2n+1$ by using the $k-$th equation. 
Note that after this operation, the coefficient of $x_{k+1}$ is unit, and the coefficients of the other variables are in $F^{1}$ in the left hand side of the $(k+1)-$th equation, that the coefficient of $x_{i}$ is a unit element of $X$, and the coefficients of the variables other than $x_{i-1}$ and $x_{i}$ are in $F^{1}$ in the left hand side of the $i-$th equation, for each $k+2\leq i\leq 2n+1$, and that in the $j-th$ equation for each $k+1\leq j\leq 2n+1$, there are no $x_{l}$ for $l<k$. 
Then multiply some element from the left to the $(k+1)-$th equation so that the coefficient of $x_{k+1}$ in the left hand side of the $(k+1)-$th equation becomes to be $1$. 
This complete the inductive step. 
Finally, we get the unique solution for our equation. 
Moreover, by construction, the solution for $x_{2n+1}$ is of the form of $F_{D}(s,u,t)\cdot a$, where $F_{D}(s,u,t)$ is an element of $X$ which is independent of $a$.

\begin{figure}[htbp]
 \begin{center}
  \includegraphics[width=120mm]{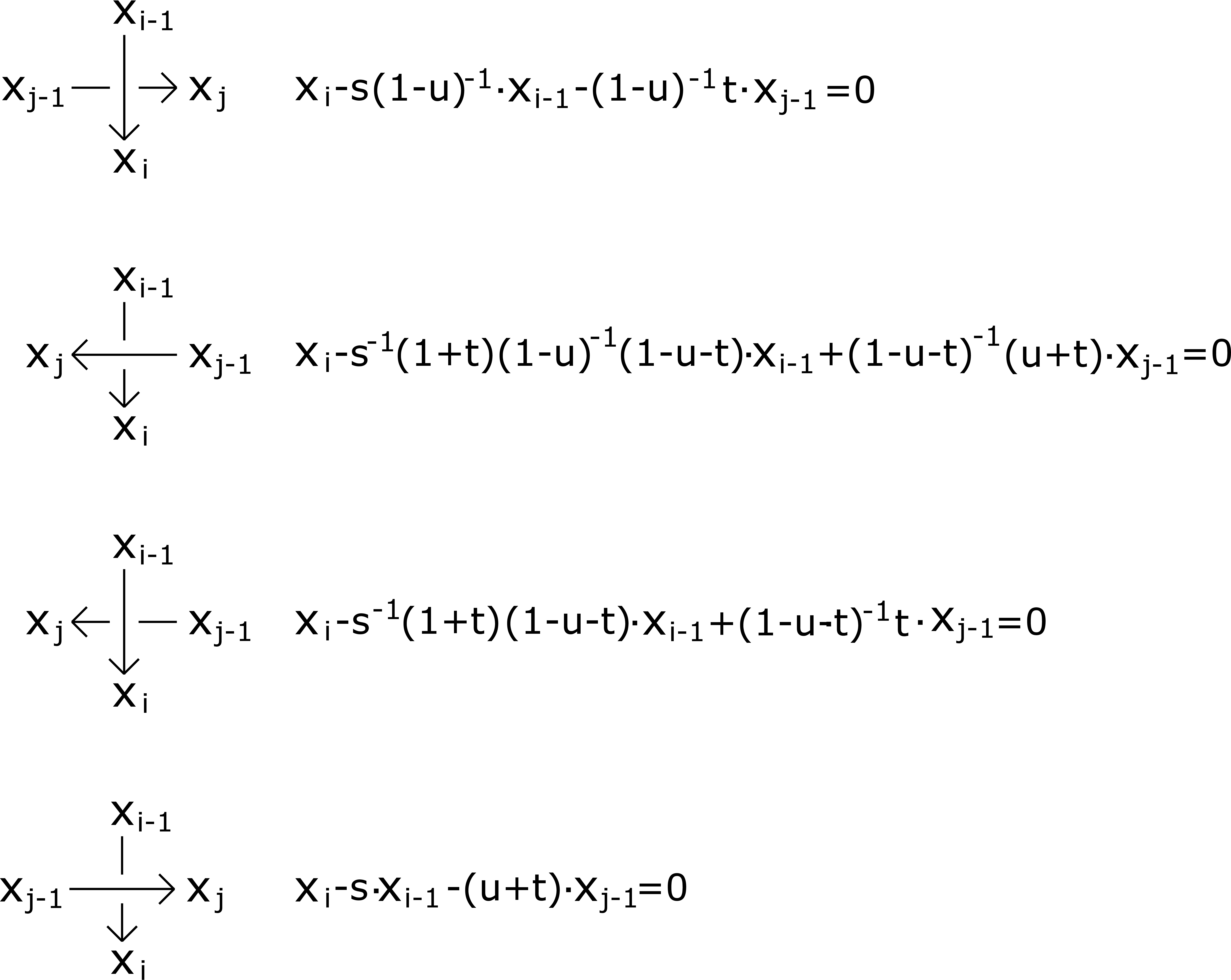}
 \end{center}
 \caption{The $i-$th equations for the types of crossings}
 \label{crossing_equation}
\end{figure}

\end{proof}

By the above, we define our invariant. 
\begin{defini}\label{defini_inv}
For a long virtual knot $K$, take a long virtual knot diagram $D$ representing $K$. 
Then we define $F_{K}(s,u,t)$ to be an element $F_{D}(s,u,t)$ of $X$. 
By Fact~\ref{fact_inv}, this is independent of the choice of $D$, therefore well-defined. 
\end{defini}

\begin{example}\label{example_computation}
We will review a process for the computation. 
Consider the long virtual knot $K_{2}(1)$ in Figure~\ref{k2(1)}. 

\begin{figure}[htbp]
 \begin{center}
  \includegraphics[width=40mm]{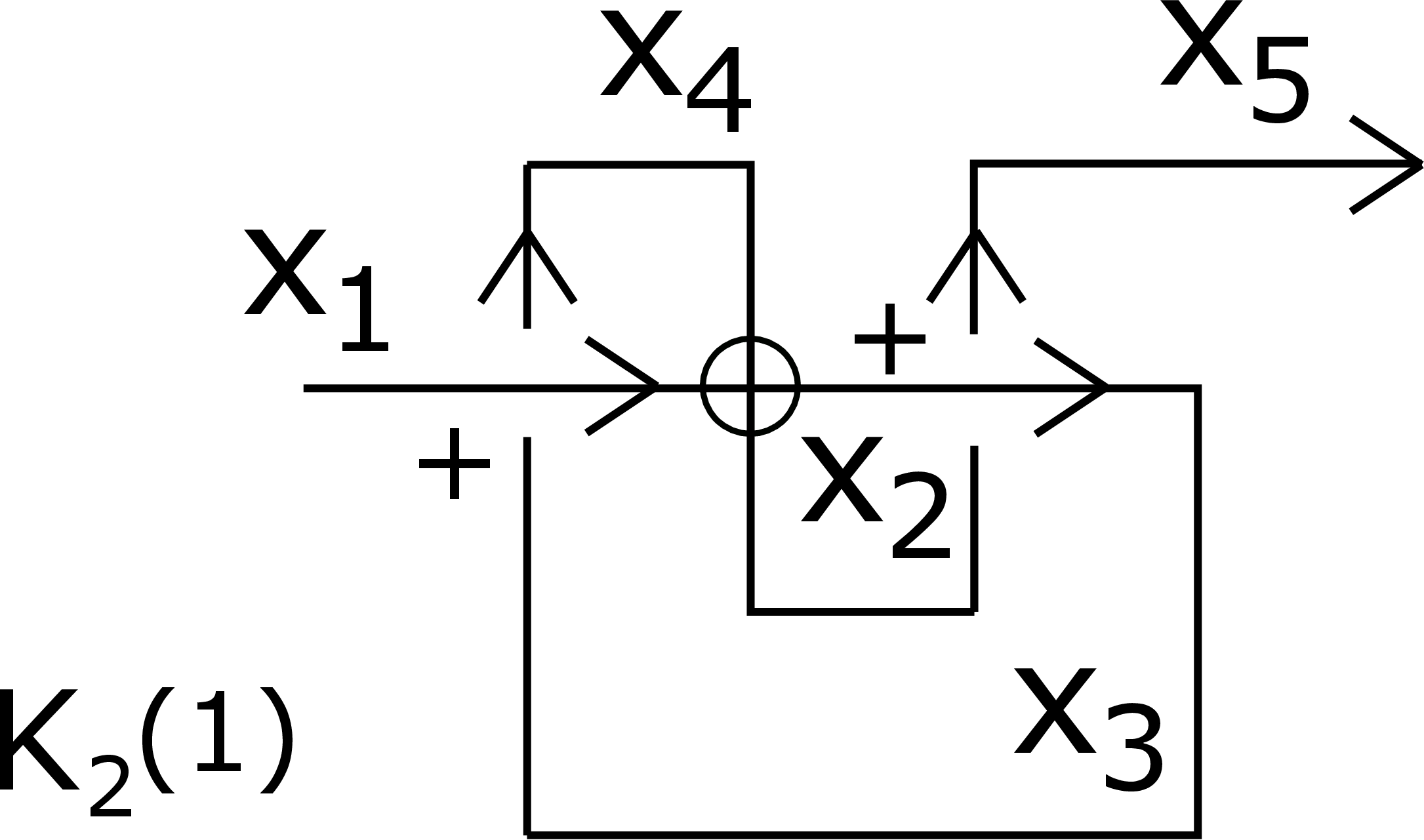}
 \end{center}
 \caption{The long virtual knot $K_{2}(1)$}
 \label{k2(1)}
\end{figure}

The linear equations are following: 
\begin{align}
  &\begin{cases}
    x_{1} = a \\
    x_{2}-s(1-u)^{-1}x_{1}-(1-u)^{-1}tx_{3} = 0\\
    x_{3}-s(1-u)^{-1}x_{2}-(1-u)^{-1}tx_{4} = 0\\
    x_{4}-s^{-1}(1+t)(1-u)^{-1}(1-u-t)x_{3}+(1-u-t)^{-1}(u+t)x_{1}=0\\
    x_{5}-s^{-1}(1+t)(1-u)^{-1}(1-u-t)x_{4}+(1-u-t)^{-1}(u+t)x_{2}=0
  \end{cases} \notag 
\end{align}
Set $A$, $B$, $C$, $D$ to be $s(1-u)^{-1}$, $(1-u)^{-1}t$, $s^{-1}(1+t)(1-u)^{-1}(1-u-t)$, $-(1-u-t)^{-1}(u+t)$, respectively due to the space limitation. 
Note that $A$ and $C$ are unit elements of $X$ and that $B$ and $D$ are in $F^{1}$.

\begin{align}
&\left(
  \begin{array}{ccccc|@{\quad}c} 
    1 & 0& 0 & 0 & 0 & a\\
    -A & 1 & -B & 0 & 0 & 0\\
    0 & -A & 1 & -B & 0 & 0\\
    -D & 0 & -C & 1 & 0 & 0\\
    0 & -D & 0 & -C & 1 & 0
  \end{array}
\right) \notag \\ 
\notag \\
\Longleftrightarrow & 
\left(
  \begin{array}{ccccc|@{\quad}c} 
    1 & 0& 0 & 0 & 0 & a\\
    0 & 1 & -B & 0 & 0 & Aa\\
    0 & -A & 1 & -B & 0 & 0\\
    0 & 0 & -C & 1 & 0 & Da\\
    0 & -D & 0 & -C & 1& 0
  \end{array}
\right) 
 \notag \\ 
 \notag \\
\Longleftrightarrow &\left(
  \begin{array}{ccccc|@{\quad}c} 
    1 & 0& 0 & 0 & 0 & a\\
    0 & 1 & -B & 0 & 0 & Aa\\
    0 & 0 & 1-AB & -B & 0 & A^{2}a\\
    0 & 0 & -C & 1 & 0 & Da\\
    0 & 0 & -DB & -C & 1 & DAa
  \end{array}
\right) \notag 
\end{align} 

\begin{align}
\Longleftrightarrow &
\left(
  \begin{array}{ccccc|@{\quad}c} 
    1 & 0& 0 & 0 & 0 & a\\
    0 & 1 & -B & 0 & 0 & Aa\\
    0 & 0 & 1 & -(1-AB)^{-1}B & 0 & (1-AB)^{-1}A^{2}a\\
    0 & 0 & -C & 1 & 0 & Da\\
    0 & 0 & -DB & -C & 1 & DAa
  \end{array}
\right) 
 \notag \\ 
\notag \\
\Longleftrightarrow &
\left(
  \begin{array}{ccccc|@{\quad}c} 
    1 & 0& 0 & 0 & 0 & a\\
    0 & 1 & -B & 0 & 0 & Aa\\
    0 & 0 & 1 & -(1-AB)^{-1}B & 0 & (1-AB)^{-1}A^{2}a\\
    0 & 0 & 0 & 1-C(1-AB)^{-1}B & 0 & \left(D+C(1-AB)^{-1}A^{2}\right)a\\
    0 & 0 & 0 & -C-DB(1-AB)^{-1}B & 1 & \left(DA+DB(1-AB)^{-1}A^{2}\right) a
  \end{array}
\right)
 \notag \\ 
\notag \\
\Longleftrightarrow &
\left(
  \begin{array}{ccccc|@{\quad}c} 
    1 & 0& 0 & 0 & 0 & a\\
    0 & 1 & -B & 0 & 0 & Aa\\
    0 & 0 & 1 & -(1-AB)^{-1}B & 0 & (1-AB)^{-1}A^{2}a\\
    0 & 0 & 0 & 1 & 0 &\left( 1-C(1-AB)^{-1}B\right)^{-1} Ea\\
    0 & 0 & 0 & -C-DB(1-AB)^{-1}B & 1 & F a
  \end{array}
\right) 
 \notag \\
& \notag \\
& \ \  \text{, where $E=D+C(1-AB)^{-1}A^{2}$ and $F=DA+DB(1-AB)^{-1}A^{2}$} \notag 
\end{align}
\begin{align}
\Longleftrightarrow &
\left(
  \begin{array}{ccccc|@{\quad}c} 
    1 & 0& 0 & 0 & 0 & a\\
    0 & 1 & -B & 0 & 0 & Aa\\
    0 & 0 & 1 & -(1-AB)^{-1}B & 0 & (1-AB)^{-1}A^{2}a\\
    0 & 0 & 0 & 1 & 0 &\left( 1-C(1-AB)^{-1}B\right)^{-1} Ea\\
    0 & 0 & 0 & 0& 1 & G a
  \end{array}
\right) 
 \notag \\ 
 & \notag \\
& \ \  \text{, where $G=\left(F+\left(C+DB(1-AB)^{-1}B\right)\left( 1-C(1-AB)^{-1}B\right)^{-1} E \right)$} \notag 
\end{align}
Then we get $F_{K_{2}(1)}(s,u,t)$ is $G$. 

\end{example}

\subsection{Some properties of $F_{K}(s,u,t)$}

By Proposition~\ref{prop_base}, we can represent $F_{K}(s,u,t)$ as $F_{K}(s,u,t)=\sum \limits_{w\in W(u,t)} f_{K,w}(s)\cdot w$ uniquely, where $f_{K,w}(s)$'s are elements of $\mathbb{Z}[s^{\pm1}]$ and $W(u,t)$ is a set defined at Definition~\ref{def_base}. 
We can compute $f_{K,\varnothing}(s)$, $f_{K,u}(s)$ and $f_{K,t}(s)$ easily as in the next two propositions. 

\begin{prop}
The equation $f_{K,\varnothing}(s)=1$ holds for every long virtual knot $K$. 
In other word, $F_{K}(s,0,0)=1$. 
\end{prop}
\begin{proof}
Take a long virtual knot diagram $D$ representing $K$. 
Consider the conditions on $X/F^{1}$ in Figure~\ref{Xcoloring}. 
A semiarc is multipled by $s$ if it passes over (or under) a positive (or negative, respectively) crossing, and is multipled by $s^{-1}$ if it passes over (or under) a negative (or positive) crossing. 
In term of the pointed Gauss diagram, semiarcs is multipled by $s$ or $s^{-1}$ if it passed the endpoint of a dashed arrow with symbol $+$ or $-$, respectively. 
During traveling $D$, each classical crossing is passed twice, once over and the other under. 
Thus the terminal semiarc is obtained by multiplying $1$ from the initial semiarc. 
\end{proof}

\begin{prop}\label{deg1}
The equations $f_{K,u}(s)=-W_{K}(s^{-1})$ and $f_{K,t}(s)=f_{K,u}(s)-f_{K,u}(s^{-1})=W_{K}(s)-W_{K}(s^{-1})$ hold for every long virtual knot $K$, where $W_{K}(s)\in \mathbb{Z}[s^{\pm1}]$ is a polynomial defined at Fact~\ref{writhe polynomial}.  
\end{prop}

\begin{proof}
We consider with $X/F^{2}$ i.e. ignore parts whose degrees are grater than one. 
Note that $X/F^{2}$ is commutative. 
In $X/F^{2}$, the coloring rules at each classical crossing are as in Figure~\ref{Xcoloring_deg1}. 

\begin{figure}[htbp]
 \begin{center}
  \includegraphics[width=80mm]{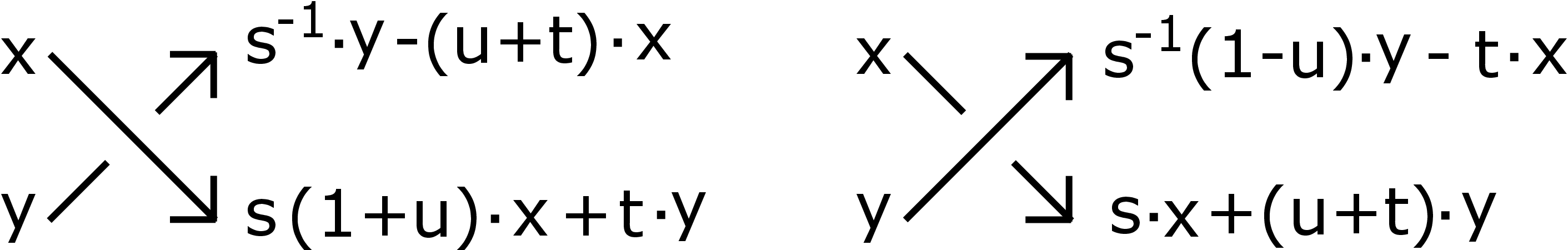}
 \end{center}
 \caption{$(X; \underline{\triangleright}, \overline{\triangleright})-$coloring at classical crossings up to degree one}
 \label{Xcoloring_deg1}
\end{figure}

For a long virtual knot diagram $K$, take the $X/F^{2}-$coloring such that the assignment of the initial arc is $1\in X/F^{2}$. 
Research the change of the assignment of the arc under passing a crossing $c$. 
Note that the degree zero part of the assignment of each arc is $s^{k}$ for some integer $k$. 
Suppose that the assignment of the arc $a$ before passing a crossing $c$ is $s^{k}(1+Au+Bt)$ for $k\in \mathbb{Z}$ and $A,B\in \mathbb{Z}[s^{\pm1}]$.  
We divide into eight cases. See Figure~\ref{eightcases}.

(1) The case where $\epsilon(c)=1$ and ${c\in \rm EO}$ and $a$ is overarc \\
In this case, the degree zero part of the underarc before passing $c$ is $s^{k+1+\alpha(c)}$. 
Then the assignment of the arc next to $a$ is $s(1+u)\cdot s^{k}(1+Au+Bt)+t\cdot s^{k+1+\alpha(c)}=s^{k+1}\left(1+(A+1)u+(B+s^{\alpha(c)})t\right)$. 

(2) The case where $\epsilon(c)=1$ and ${c\in \rm EO}$ and $a$ is underarc \\
In this case, the degree zero part of the overarc before passing $c$ is $s^{k-1-\alpha(c)}$. 
Then the assignment of the arc next to $a$ is $s^{-1}\cdot s^{k}(1+Au+Bt)-(u+t)\cdot s^{k-1-\alpha(c)}=s^{k-1}\left(1+(A-s^{-\alpha(c)})u+(B-s^{-\alpha(c)})t\right)$. 

(3) The case where $\epsilon(c)=-1$ and ${c\in \rm EO}$ and $a$ is underarc \\
In this case, the degree zero part of the overarc before passing $c$ is $s^{k+1-\alpha(c)}$. 
Then the assignment of the arc next to $a$ is $s\cdot s^{k}(1+Au+Bt)+(u+t)\cdot s^{k+1-\alpha(c)}=s^{k+1}\left(1+(A+s^{-\alpha(c)})u+(B+s^{-\alpha(c)})t\right)$. 

(4) The case where $\epsilon(c)=-1$ and ${c\in \rm EO}$ and $a$ is overarc \\
In this case, the degree zero part of the underarc before passing $c$ is $s^{k-1+\alpha(c)}$. 
Then the assignment of the arc next to $a$ is $s^{-1}(1-u)\cdot s^{k}(1+Au+Bt)-t\cdot s^{k-1+\alpha(c)}=s^{k-1}\left(1+(A-1)u+(B-s^{\alpha(c)})t\right)$. 

(5) The case where $\epsilon(c)=1$ and ${c\in \rm EU}$ and $a$ is overarc \\
In this case, the degree zero part of the underarc before passing $c$ is $s^{k+1-\alpha(c)}$. 
Then the assignment of the arc next to $a$ is $s(1+u)\cdot s^{k}(1+Au+Bt)+t\cdot s^{k+1-\alpha(c)}=s^{k+1}\left(1+(A+1)u+(B+s^{-\alpha(c)})t\right)$. 

(6) The case where $\epsilon(c)=1$ and ${c\in \rm EU}$ and $a$ is underarc \\
In this case, the degree zero part of the overarc before passing $c$ is $s^{k-1+\alpha(c)}$. 
Then the assignment of the arc next to $a$ is $s^{-1}\cdot s^{k}(1+Au+Bt)-(u+t)\cdot s^{k-1-\alpha(c)}=s^{k-1}\left(1+(A-s^{\alpha(c)})u+(B-s^{\alpha(c)})t\right)$. 

(7) The case where $\epsilon(c)=-1$ and ${c\in \rm EU}$ and $a$ is underarc \\
In this case, the degree zero part of the overarc before passing $c$ is $s^{k+1+\alpha(c)}$. 
Then the assignment of the arc next to $a$ is $s\cdot s^{k}(1+Au+Bt)+(u+t)\cdot s^{k+1+\alpha(c)}=s^{k+1}\left(1+(A+s^{\alpha(c)})u+(B+s^{\alpha(c)})t\right)$. 

(8) The case where $\epsilon(c)=-1$ and ${c\in \rm EU}$ and $a$ is overarc \\
In this case, the degree zero part of the underarc before passing $c$ is $s^{k-1-\alpha(c)}$. 
Then the assignment of the arc next to $a$ is $s^{-1}(1-u)\cdot s^{k}(1+Au+Bt)-t\cdot s^{k-1-\alpha(c)}=s^{k-1}\left(1+(A-1)u+(B-s^{-\alpha(c)})t\right)$. 

By noting that each crossing is passed twice from over and under, we can compute that
\begin{align}
 &f_{K,u}(s)=\sum \limits_{c\in {\rm EO}}\epsilon(c) (1-s^{-\alpha(c)})  +  \sum \limits_{d\in {\rm EU}} \epsilon(d) (1-s^{\alpha(d)}) = -W_{K}(s^{-1}) \notag \\
 & f_{K,t}(s)=\sum \limits_{c\in {\rm EO}}\epsilon(c) (s^{\alpha(c)}-s^{-\alpha(c)})  +  \sum \limits_{d\in {\rm EU}} \epsilon(d) (s^{-\alpha(d)}-s^{\alpha(d)})=f_{K,u}(s)-f_{K,u}(s^{-1}) \notag
\end{align}

\begin{figure}[htbp]
 \begin{center}
  \includegraphics[width=160mm]{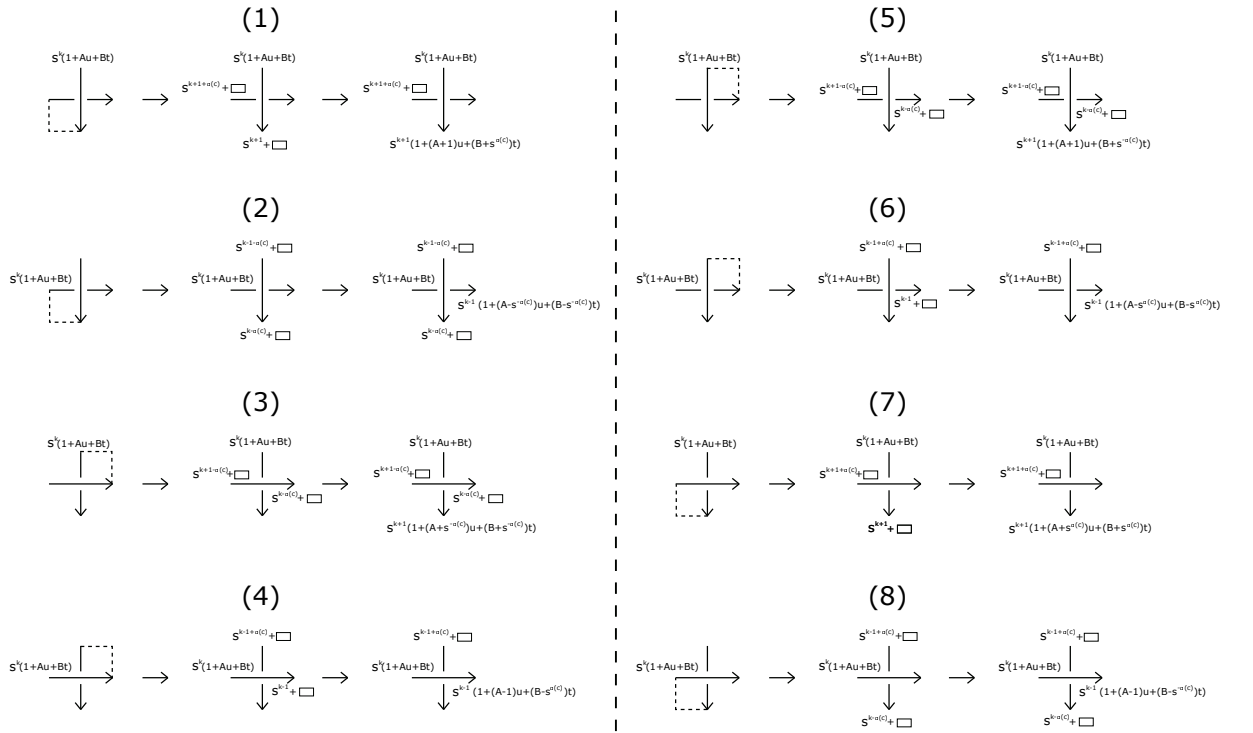}
 \end{center}
 \caption{The eight cases, where boxes represent degree one parts}
 \label{eightcases}
\end{figure}

\end{proof}

Next proposition states that our invariant is an anti-homomorphism from the set of long virtual knots to the ring $X$.
\begin{prop}\label{product}
Let $K_{1}$ and $K_{2}$ be two long virtual knots. 
Then $F_{K_{1}\cdot K_{2}}(s,u,t)=F_{K_{2}}(s,u,t) \cdot F_{K_{1}}(s,u,t)$ holds. 
\end{prop}
\begin{proof}
Take long virtual knot diagrams $D_1$ and $D_2$ representing $K_1$ and $K_2$, respectively. 
Note that an input element of $X$, which is the assignment of the initial semiarc of $D_{i}$ of $(X; \underline{\triangleright}, \overline{\triangleright})-$coloring, is multiplied by $F_{K_{i}}(s,u,t)$ from the left to be the output element, which is the assignment of the terminal semiarc of $D_{i}$ of $(X; \underline{\triangleright}, \overline{\triangleright})-$coloring, for each $i=1,2$. 
Therefore, the input element of $D_{1}\cdot D_{2}$ is multiplied by $F_{K_{2}}(s,u,t)\cdot F_{K_{1}}(s,u,t)$ from the left to be the output element. 
\end{proof}

\begin{prop}
For a classical long virtual knot $K$, the equation $F_{K}(s,u,t)=1$ holds.  
\end{prop}
\begin{proof}
Take a long virtual knot diagram $D$ without virtual crossongs representing $K$. 
Suppose that $D$ has $m$ classical crossings. 
Give them the name $c_1,\dots, c_m$, and give the semiarcs of $D$ variables $x_1,\dots, x_{2m+1}$. 
We will show that the $(X; \underline{\triangleright}, \overline{\triangleright})-$coloring conditions force $x_{2m+1}$ to be equal to $x_1$. 
At each classical crossing $c_k$, there are two equations for $(X; \underline{\triangleright}, \overline{\triangleright})-$coloring. 
See Figure~\ref{relation_Rk}. 
We set $O_k$ and $U_k$ to be the conditions $w-s(1-u)^{-1}x-(1-u)^{-1}ty=0$ and $z-s^{-1}(1+t)(1-u)^{-1}(1-u-t)y+(1-u-t)^{-1}(u+t)x=0$, respectiely if $c_k$ is a positive crossing, where $x,y,z,w$ are variables around $c_k$ as in Figure~\ref{relation_Rk}. 
And we set $O_k$ and $U_k$ to be the conditions $z-s^{-1}(1+t)(1-u-t)y+(1-u-t)^{-1}x=0$ and $w-sx-(u+t)y=0$, respectiely if $c_k$ is a negative crossing, where $x,y,z,w$ are variables around $c_k$ as in Figure~\ref{relation_Rk}. 
Then set $R_k$ to be the condition $s^{-1}(1+t)O_k+U_k$ if $c_k$ is a positive crossing, and to be the condition $O_k+s^{-1}(1+t)U_k$ if $c_k$ is a negative crossing. 
In any cases, the condition $R_k$ is $z+s^{-1}(1+t)w-x-s^{-1}(1+t)y=0$. 

\begin{figure}[htbp]
 \begin{center}
  \includegraphics[width=120mm]{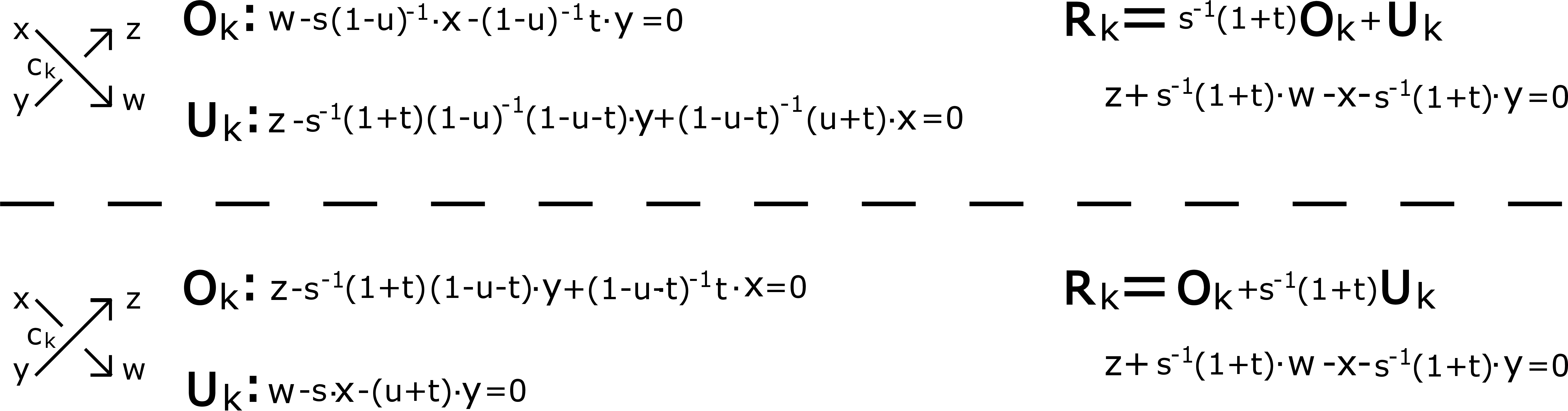}
 \end{center}
 \caption{The relation $R_{k}$ definied at a classical crossing $c_{k}$}
 \label{relation_Rk}
\end{figure}

Next we define the weights at classical crossings via the weights of the regions. 
Since $D$ is an immersion of $\mathbb{R}$ in $\mathbb{R}^{2}$, the plane $\mathbb{R}^{2}$ is divided into $m+2$ regions. 
Assign integers to these regions with the following rule: 
If some region is reached by passing over a semiarc from the left from the region whose assignment is $n$, then the assignment of the region is $n+1$. 
See Figure~\ref{weight_region}. 
This assignment is called a weight of regions. 
Take (any) weight of regions. 
Using this weight of regions, we define the weight of the crossing $c_k$ as the assignment of one of the adjacent region as in Figure~\ref{weight_crossing}. 

\begin{figure}[htbp]
 \begin{center}
  \includegraphics[width=80mm]{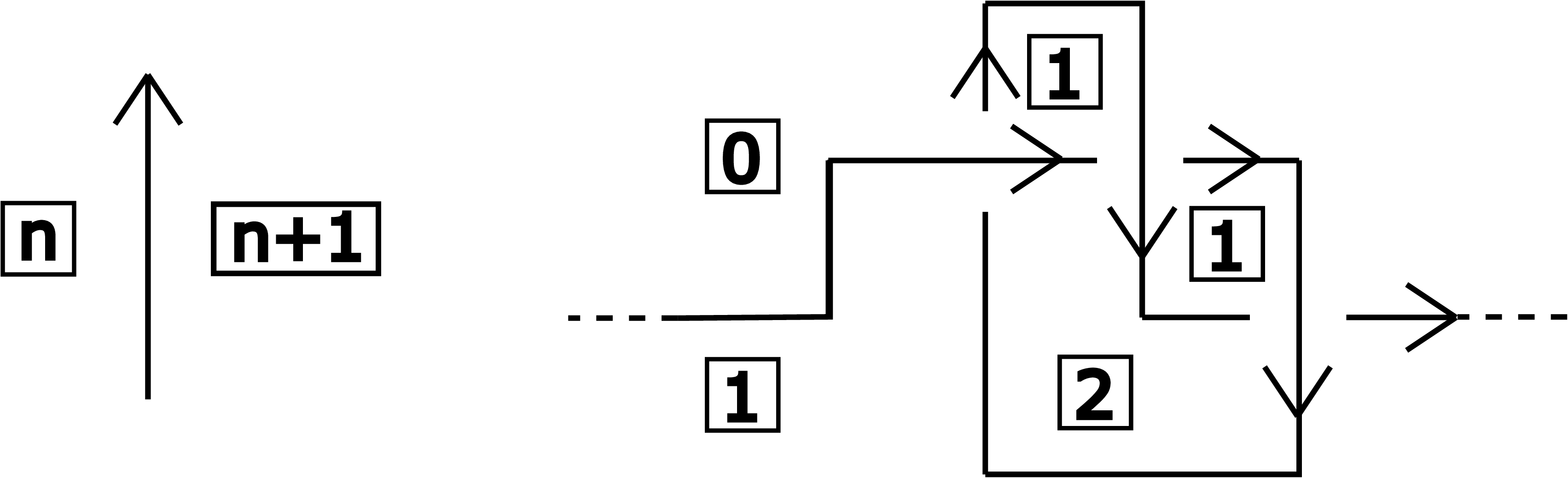}
 \end{center}
 \caption{The rule for weights of regions and an example}
 \label{weight_region}
\end{figure}

\begin{figure}[htbp]
 \begin{center}
  \includegraphics[width=40mm]{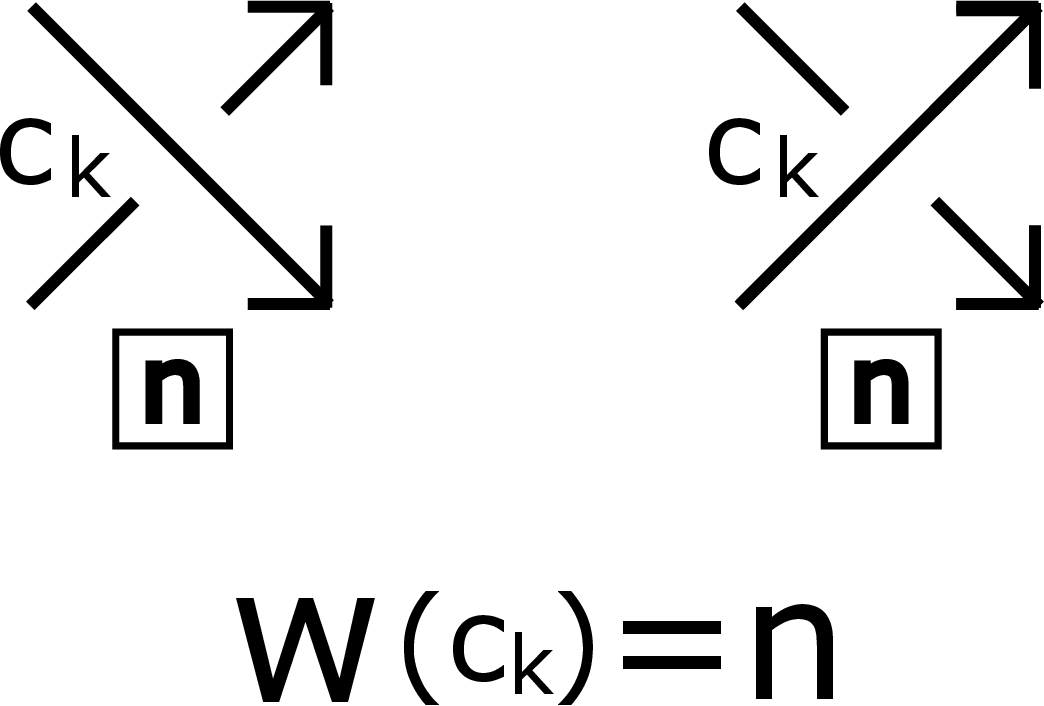}
 \end{center}
 \caption{The weight $w(c_{k})$ of a crossing $c_{k}$}
 \label{weight_crossing}
\end{figure}

Then consider the condition $R=\sum \limits_{i=1}^{m} \left(s^{-1}(1+t)\right)^{w(c_i)}\cdot R_i$. 
By computation, we see that the coefficients of all $x_{i}$ other than $x_1$ and $x_{2m+1}$ in $R$ is zero. See Figure~\ref{Rk_connect}. 
In Figure~\ref{Rk_connect}, all crossings are depicted as positive crossings, but this is enough for general cases since $R_k$ is independent of the sign of $c_k$. 
Similarly, we see that  the coefficient of $x_{1}$ in $R$ is some unit and that of $x_{2m+1}$ in $R$ is as same as this unit with opposite sign. See Figure~\ref{Rk_endpoints}. 
Hence we see that $x_{2m+1}=x_1$. This means $F_{K}(s,u,t)=1$. 

\begin{figure}[htbp]
 \begin{center}
  \includegraphics[width=145mm]{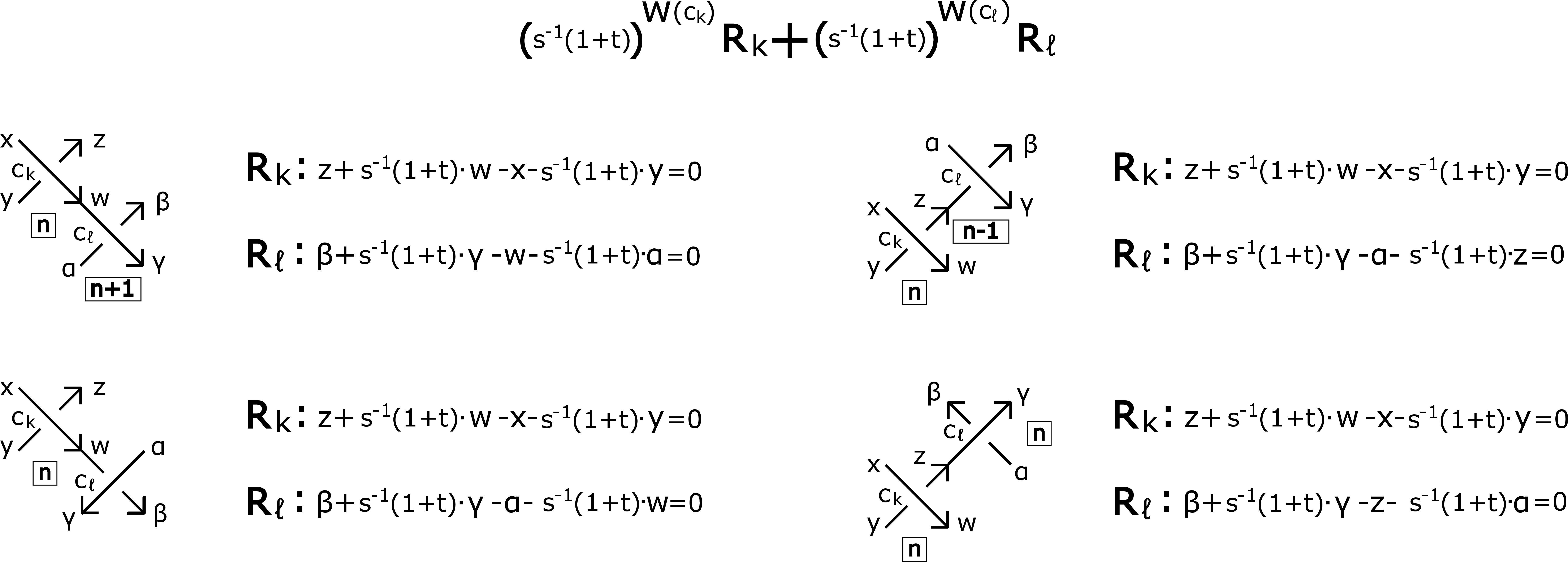}
 \end{center}
 \caption{Vanishment of the coefficients of the intermediate semiarcs in $\sum \limits_{i=1}^{m} \left(s^{-1}(1+t)\right)^{w(c_{i})}R_{i}$}
 \label{Rk_connect}
\end{figure}

\begin{figure}[htbp]
 \begin{center}
  \includegraphics[width=145mm]{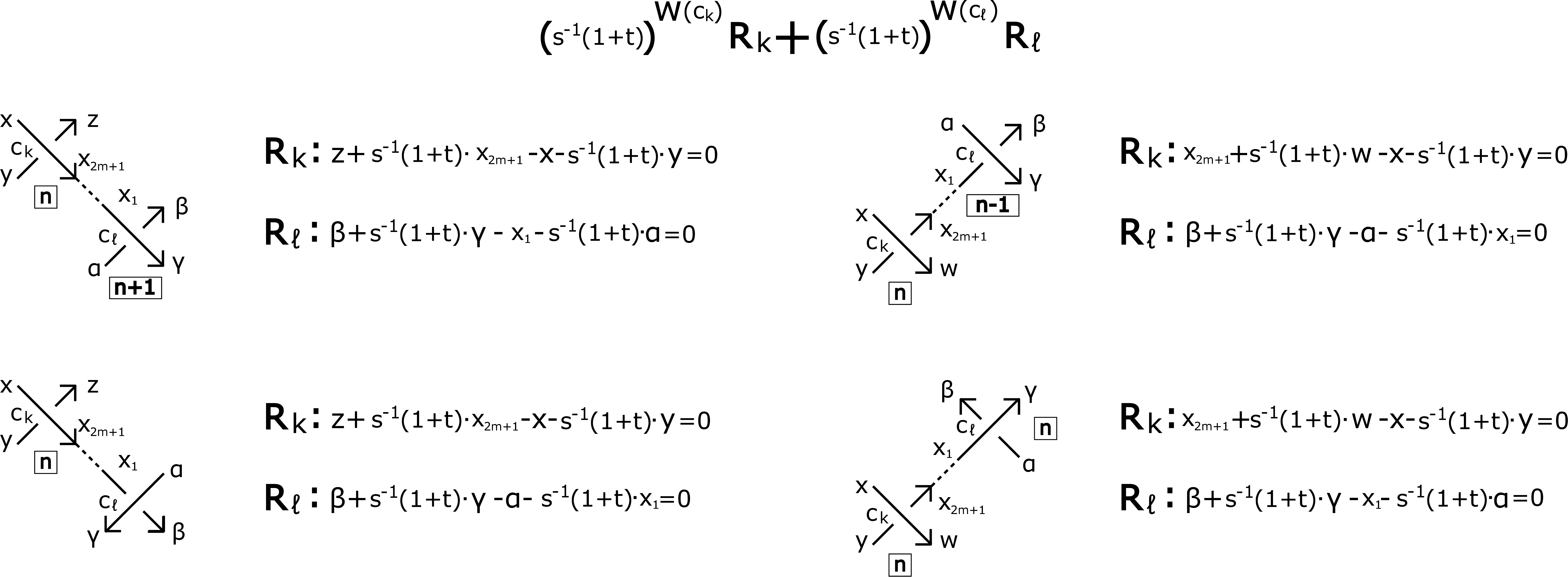}
 \end{center}
 \caption{The coefficients of the initial and terminal sumiarcs in $\sum \limits_{i=1}^{m} \left(s^{-1}(1+t)\right)^{w(c_{i})}R_{i}$}
 \label{Rk_endpoints}
\end{figure}

\end{proof}

\begin{prop}
For a long virtual knot $K$, we have $F_{-K^{*}}(s,u,t)=\left(F_{K}(s,u,t)\right)^{-1}$. 
\end{prop}

\begin{proof}
Take a long virtual knot diagram $D$ representing $K$. 
Suppose that $D$ has $n$ classical crossings. 
Assign variables $x_{1},\cdots, x_{2n+1}$ to semi-arcs in order from the initial arc to the terminal arc. 
Take the equations at each classical crossings for $X-$coloring to compute $F_{K}(s,u,t)$. 
Consider $-D^{*}$. 
Assign variables $y_{i}$ to the semiarc of $-D^{*}$ corresponding to the semiarc of $D$ whose assigned variable is $x_{i}$ for $i=1,\cdots,2n+1$. 
Note that the indices assigned to semiarcs of $-D^{*}$ are reverse order from the initial arc to the terminal arc. 
By computations, we see that the two equations at a classical crossing $c$ of $-D^{*}$ are equivalent to those of at the classical crossing of $D$ corresponding to $c$ as in Figure~\ref{coloring_inverse}. 
This can be checked by showing that matrices of $X$ coefficients 
$\left(
  \begin{array}{cc} 
    s^{-1}(1+t)(1-u-t) & -(1-u-t)^{-1} t\\
    u+t & s
  \end{array}
\right) $
 and 
$\left(
  \begin{array}{cc} 
    s(1-u)^{-1} & (1-u)^{-1}t \\
    -(1-u-t)^{-1}(u+t) & s^{-1}(1+t)(1-u)^{-1}(1-u-t)
  \end{array}
\right) $ 
are inverses of each other. 
Hence we have $y_{2n+1}=F_{K}(s,u,t)\cdot y_{1}$ as $x_{2n+1}=F_{K}(s,u,t)\cdot x_{1}$. 
Since $y_{1}$ is the variable assigned to the terminal arc of $-D^{*}$ and $y_{2n+1}$ is the variable assigned to the initial arc of $-D^{*}$, we get $F_{-K^{*}}(s,u,t)=\left(F_{K}(s,u,t)\right)^{-1}$. 

\begin{figure}[htbp]
 \begin{center}
  \includegraphics[width=130mm]{coloring_inverse.eps}
 \end{center}
 \caption{Equations of classical crossings of $D$ and $-D^{*}$}
 \label{coloring_inverse}
\end{figure}

\end{proof}

\begin{prop}\label{base change}
Let $K$ and $K'$ be two long virtual knots such that the closures of $K$ and $K'$ are equivalent virtual knots. 
Then there are two unit elements $\alpha$ and $\beta$ of $X$ such that $F_{K'}(s,u,t)=1+\alpha \left(F_{K}(s,u,t)-1\right)\beta$ holds. 
Inparticular, $F_{K}(s,u,t)=1$ if and only if $F_{K'}(s,u,t)=1$ holds. 
\end{prop}

\begin{proof}
Consider using pointed Gauss diagrams. 
It is enough to show the statement under sliding the specified point along one classical crossing point. 
Before sliding, construct and solve the system of equations. 
Set the solution to be $x_{i}=A_{i}x_{1}$ for $i=1,\dots,2n+1$ for unit elements $A_{i}\in X$. 
Note that $A_{1}=1$ and $A_{2n+1}=F_{K}(s,t,u)$ holds. 
Set the first equation to be $x_{2}=Bx_{1}+Cx_{k}$ for some integer $k$, a unit element $B\in X$ and $C\in X$ of degree at least $1$. 
Note that $A_{2}=B+CA_{k}$ holds. 
After sliding, the system of equations are changed as follows: 
discard $x_{2}=Bx_{1}+Cx_{k}$, replace $x_{1}$ in the remaining equations with $x_{2n+1}$, and introduce an equation $y=Bx_{1}+Cx_{k}$, where $y$ is the variable assigned to the new terminal arc. 
See Figure~\ref{basechange}. 
Then we can solve this new system of equations as $x_{i}=A_{i}\left(A_{2}\right)^{-1}$ for $i=2,\dots,2n+1$ and $y=\left(BA_{2n+1}\left(A_{2}\right)^{-1}+CA_{k}\left(A_{2}\right)^{-1}\right)x_{2}=\left( BA_{2n+1}\left(A_{2}\right)^{-1}+\left(A_{2}-B\right)\left(A_{2}\right)^{-1}\right)x_{2}=\left( BA_{2n+1}\left(A_{2}\right)^{-1}+1-B\left(A_{2}\right)^{-1}\right)x_{2}=\left( 1+B\left(F_{K}(s,u,t)-1\right)\left(A_{2}\right)^{-1}\right)$. 
Hence we get the statement. 

\begin{figure}[htpb]
 \begin{center}
  \includegraphics[width=110mm]{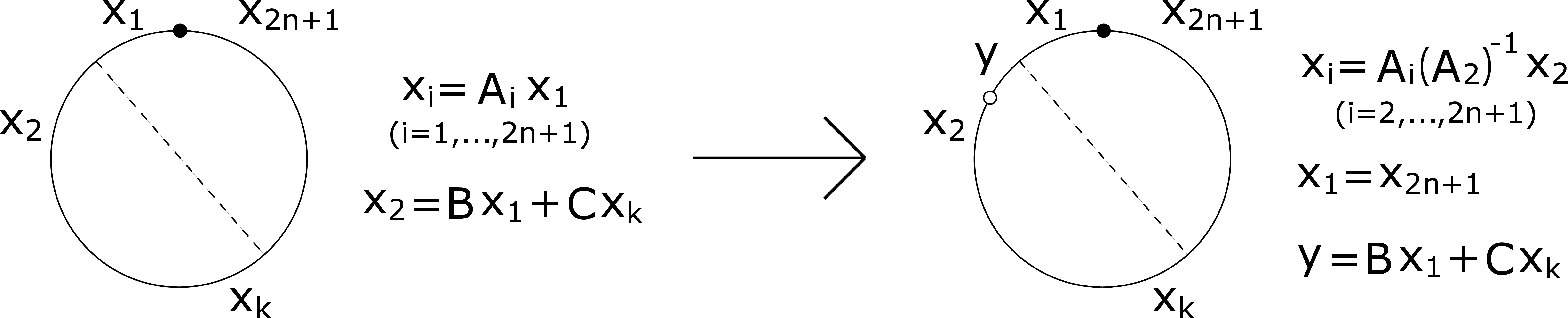}
 \end{center}
 \caption{Change of a system of equations under change of specified point, where black dot indicates old specified point and white dot indicates new one}
 \label{basechange}
\end{figure}

\end{proof}

\section{Non-commutativity of long virtual knots} \label{section5}
Being different from classical long knots, long virtual knots are non-commutative under the concatenation. 
This is firstly observed by Manturov \cite{manturov}. 
There are many works for the non-commutativity of long virtual knots. 
For example, it is proved by Chrisman \cite{chrisman} that non-equivalent two prime non-classical long virtual knots never commute. 
Since our invariant behaves as anti-homomorphism (Proposition~\ref{product}) and our ring $X$ is non-commutative, our invariant may be able to be used for checking the non-commutativity. 

\begin{thm}\label{commute}
Suppose that two long virtual knots $K_{1}$ and $K_{2}$ commute. 
Then the following two hold: 
\begin{itemize}
\item $W_{K_{1}}(s)\cdot W_{K_{2}}(s^{-1})=W_{K_{1}}(s^{-1})\cdot W_{K_{2}}(s)$ \\
\item $W_{{\rm EO}(K_{1})}(s)\cdot W_{{\rm EO}(K_{2})}(s^{-1})=W_{{\rm EO}(K_{1})}(s^{-1})\cdot W_{{\rm EO}(K_{2})}(s)$
\end{itemize}
\end{thm}

\begin{proof}
We work on $X/F^{3}$. 
Represent $F_{K_{i}}(s,t,u)$ as $F_{K_{i}}(s,t,u)=1+f_{K_{i},u}(s)u+f_{K_{i},t}(s)t+f_{K_{i},u^{2}}(s)u^{2}+f_{K_{i},ut}(s)ut+f_{K_{i},tu}(s)tu+f_{K_{i},t^{2}}(s)t^{2}$ for $i=1,2$. 
Then by 
\begin{align}
  &F_{K_{i}\cdot K_{j}}(s,t,u) \notag \\
=&F_{K_{j}}(s,t,u)\cdot F_{K_{i}}(s,t,u) \notag \\
=&\left( 1+f_{K_{j},u}(s)u+f_{K_{j},t}(s)t+f_{K_{j},u^{2}}(s)u^{2}+f_{K_{j},ut}(s)ut+f_{K_{j},tu}(s)tu+f_{K_{j},t^{2}}(s)t^{2} \right) \notag \\
& \hspace{0.5cm} \cdot \left( 1+f_{K_{i},u}(s)u+f_{K_{i},t}(s)t+f_{K_{i},u^{2}}(s)u^{2}+f_{K_{i},ut}(s)ut+f_{K_{i},tu}(s)tu+f_{K_{i},t^{2}}(s)t^{2} \right)  \notag \\
=&1+\left( f_{K_{i},u}(s)+f_{K_{j},u}(s) \right)u+\left( f_{K_{i},t}(s)+f_{K_{j},t}(s) \right)t \notag \\
& \hspace{0.5cm}+\left(f_{K_{i},u}(s)f_{K_{j},u}(s)+ f_{K_{i},u^{2}}(s)+f_{K_{j},u^{2}}(s) \right)u^{2}    \notag \\
& \hspace{0.5cm}+\left(f_{K_{i},t}(s)f_{K_{j},u}(s)+ f_{K_{i},ut}(s)+f_{K_{j},ut}(s) \right)ut  \notag  \\  
& \hspace{0.5cm}+\left(f_{K_{i},u}(s)f_{K_{j},t}(s)+ f_{K_{i},tu}(s)+f_{K_{j},tu}(s) \right)tu  \notag  \\
& \hspace{0.5cm}+\left(f_{K_{i},t}(s)f_{K_{j},t}(s)+ f_{K_{i},t^{2}}(s)+f_{K_{j},t^{2}}(s) \right)t^{2},  \notag
\end{align}
we see that $f_{K_{1},u}(s)f_{K_{2},t}(s)=f_{K_{2},u}(s)f_{K_{1},t}(s)$ must hold. 
By Proposition~\ref{deg1}, this equation becomes 
\begin{align}
&f_{K_{1},u}(s)f_{K_{2},t}(s)=f_{K_{2},u}(s)f_{K_{1},t}(s) \notag \\
\Longleftrightarrow &f_{K_{1},u}(s)\left(f_{K_{2},u}(s)-f_{K_{2},u}(s^{-1})\right)=f_{K_{2},u}(s)\left(f_{K_{1},u}(s)-f_{K_{1},u}(s^{-1})\right) \notag \\
\Longleftrightarrow &f_{K_{1},u}(s)f_{K_{2},u}(s^{-1})=f_{K_{2},u}(s)f_{K_{1},u}(s^{-1})  \notag \\
\Longleftrightarrow &W_{K_{1}}(s^{-1})\cdot W_{K_{2}}(s)=W_{K_{1}}(s)\cdot W_{K_{2}}(s^{-1}). \notag
\end{align}

Since the projection ${\rm EO}(\cdot)$ preserves the products, two long virtual knots ${\rm EO}(K_{1})$ and ${\rm EO}(K_{2})$ commute. 
Thus $W_{{\rm EO} (K_{1})}(s^{-1})\cdot W_{{\rm EO} (K_{2})}(s)=W_{{\rm EO}(K_{1})}(s)\cdot W_{{\rm EO}(K_{2})}(s^{-1})$ also holds. 
\end{proof}

\begin{rmk}
In \cite{silver}, Silver and Williams reproved the non-commutativity of two long virtual knots $K_{2}(3)^{*}$ and $K_{2}(3)^{\#*}$ which was firstly proved by Manturov, by using {\it the extended Alexander group systems}, which can be regarded as a non-commutativization of the Alexander biquandle in some sense. 
Though our biquandle $X$ also can be regarded as a non-commutativization as stated in Remark~\ref{substitute}, our invariant $F_{K}(s,u,t)$ cannot detect non-commutativity of $K_{2}(3)^{*}$ and $K_{2}(3)^{\#*}$. 
Since the closures of $K_{2}(3)^{*}$ and $K_{2}(3)^{\#*}$ are trivial virtual knots, circles with no classical and virtual crossings, the values of these under $F_{K}(s,u,t)$ are $1$'s by Proposition~\ref{base change}. 
And we see that ${\rm EO}\left(K_{2}(3)^{*}\right)={\rm EO}\left(K_{2}(3)^{\#*}\right)=K_{2}(1)^{*}$. 
Thus we have $F_{K_{2}(3)^{*}\cdot K_{2}(3)^{\#*}}(s,u,t)=F_{K_{2}(3)^{\#*}\cdot K_{2}(3)^{*}}(s,u,t)=1$ and $F_{{\rm EO}\left(K_{2}(3)^{*}\right)\cdot{\rm EO}\left(K_{2}(3)^{\#*}\right)}=F_{{\rm EO}\left(K_{2}(3)^{\#*}\right)\cdot{\rm EO}\left(K_{2}(3)^{*}\right)}=\left(F_{K_{2}(1)}(s,u,t)\right)^{2}$. 
\end{rmk}

\begin{cor}\label{center}
Suppose that a long virtual knot $K$ commutes with all long virtual knots. 
Then $W_{K}(s)$ and $W_{{\rm EO}(K)}(s)$ vanish. 
\end{cor}
\begin{proof}
First, consider the long virtual knot $K_{2}(1)$. 
Note that $K_{2}(1)\in \mathcal{EO}$. 
We see that $W_{K_{2}(1)}(s)=W_{{\rm EO}( K_{2}(1))}(s)=s-2+s^{-1}$ holds by computation. 
Note that $W_{K_{2}(1)}(s)$ is reciprocal i.e. $W_{K_{2}(1)}(s)=W_{K_{2}(1)}(s^{-1})$ holds. 
By applying Theorem~\ref{commute} to $K$ and $K_{2}(1)$, we see that $W_{K}(s)$ and $W_{{\rm EO}( K)}(s)$ are reciprocal. 
Thus the conculusion of Theorem~\ref{commute} becomes 
\begin{itemize}
\item $W_{K}(s)\cdot \left( W_{K'}(s)-W_{K'}(s^{-1})\right)=0$
\item $W_{{\rm EO}(K)}(s)\cdot \left( W_{{\rm EO}(K')}(s)-W_{{\rm EO}(K')}(s^{-1})\right)=0$
\end{itemize}
for every long virtual knots $K'$. 

Next, consider the long virtual knot $K_{3}(4)$. 
Note that $K_{3}(4)\in \mathcal{EO}$. 
We see that $W_{K_{3}(4)}(s)=W_{{\rm EO}( K_{3}(4))}(s)=s-1-s^{-1}+s^{-2}$ holds by computation. 
Note that $W_{K_{3}(4)}(s)$ is not reciprocal. 
Thus $W_{K_{3}(4)}(s)-W_{K_{3}(4)}(s^{-1})\neq 0$ holds. 
Hence we can conclude that $W_{K}(s)=0$ and $W_{{\rm EO}(K)}(s)=0$. 
\end{proof}

Consider the long virtual knot $K$ which admits a long virtual knot diagram with at most three classical crossings. 
If $K$ admits the condition of Corollary~\ref{center}, then $K$ is one of $K_{0}(1)$, $K_{3}(13)$ and $K_{3}(13)^{\#}$, all of which are classical. 
It is known that every classical long virtual knot commutes with all long virtual knots. 
Ofcourse, this can be proved by the result of Chrisman \cite{chrisman} if we know that long virtual knots whose crossing numbers are at most three are prime. 
However, the author does not know whether these are prime or not.

\vspace{0.5cm}

\ \ E-mail adress: \texttt{nnsekino@gmail.com}

\begin{figure}[htbp]
 \begin{center}
  \includegraphics[width=120mm]{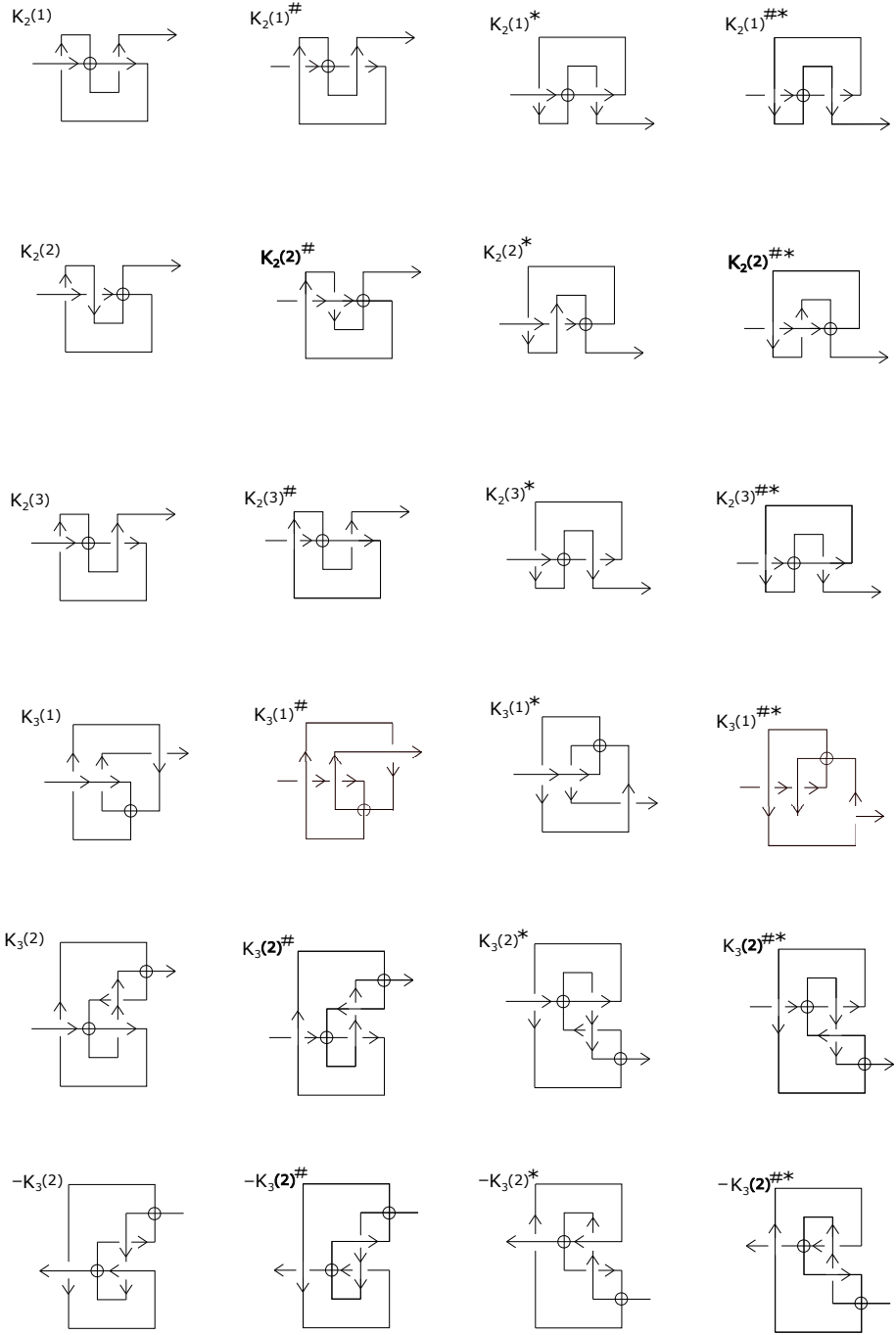}
 \end{center}
 \caption{Long virtual knots 1/10}
 \label{computation_symmetry1}
\end{figure}

\begin{figure}[htbp]
 \begin{center}
  \includegraphics[width=120mm]{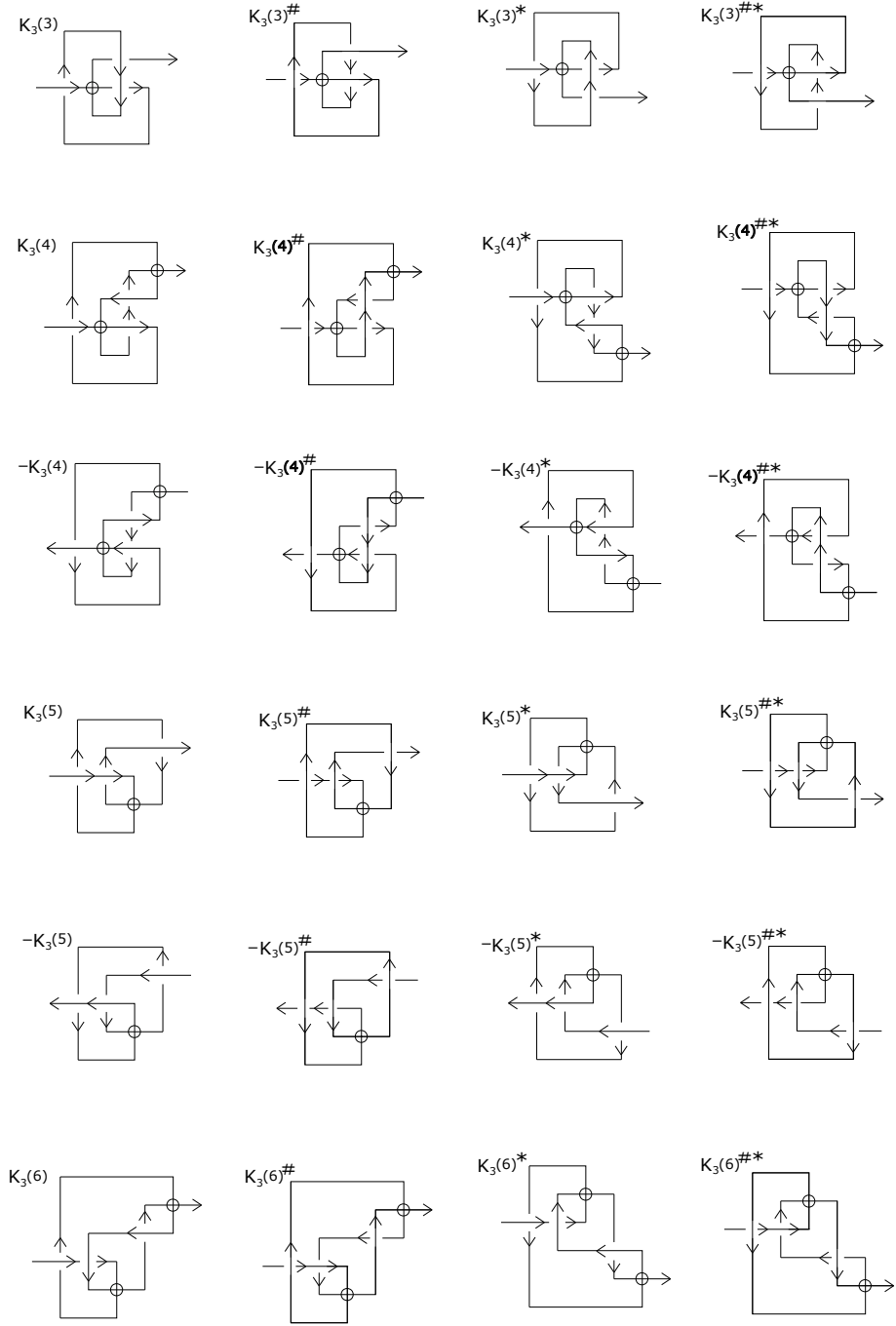}
 \end{center}
 \caption{Long virtual knots 2/10}
 \label{computation_symmetry2}
\end{figure}

\begin{figure}[htbp]
 \begin{center}
  \includegraphics[width=120mm]{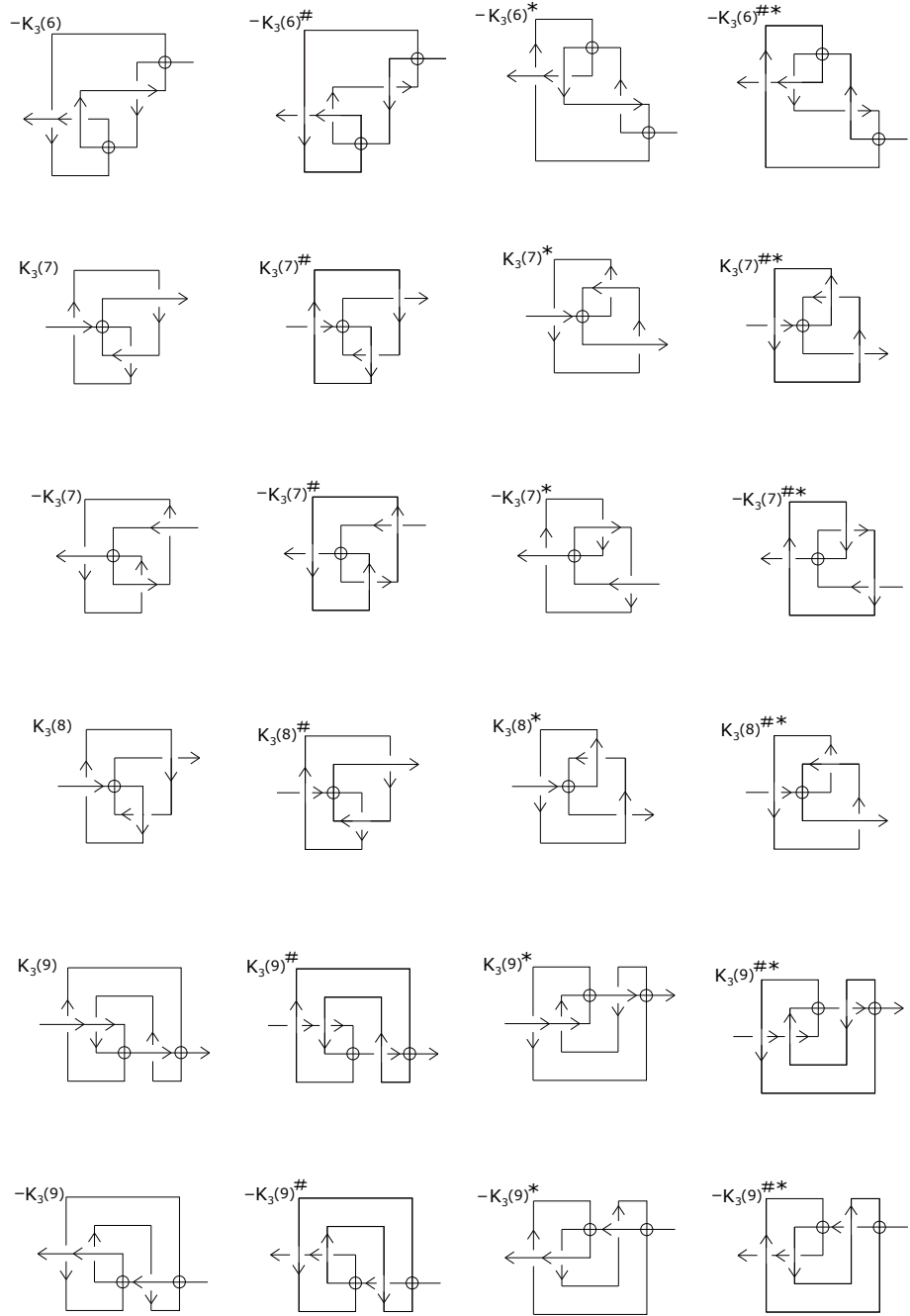}
 \end{center}
 \caption{Long virtual knots 3/10}
 \label{computation_symmetry3}
\end{figure}

\begin{figure}[htbp]
 \begin{center}
  \includegraphics[width=120mm]{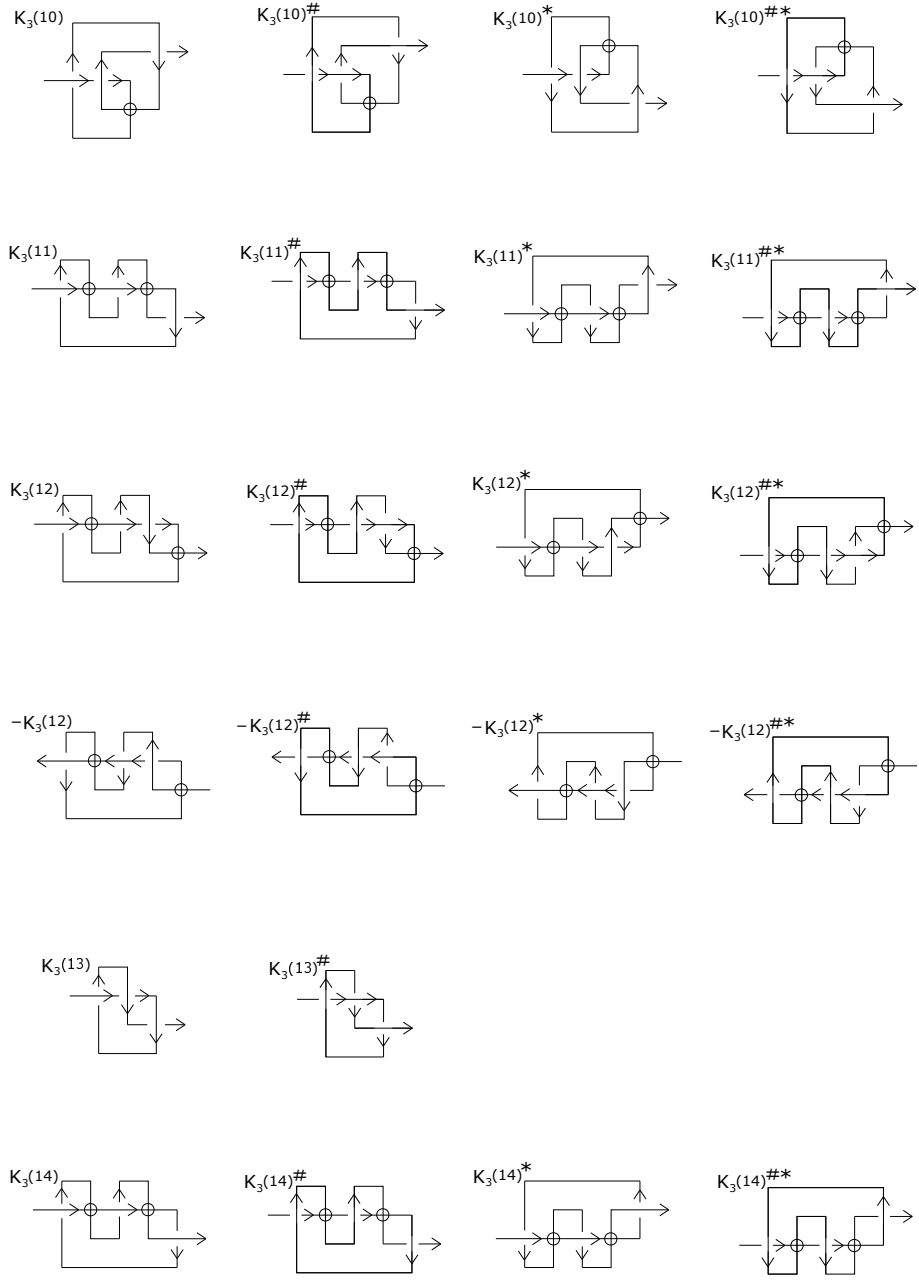}
 \end{center}
 \caption{Long virtual knots 4/10}
 \label{computation_symmetry4}
\end{figure}

\begin{figure}[htbp]
 \begin{center}
  \includegraphics[width=120mm]{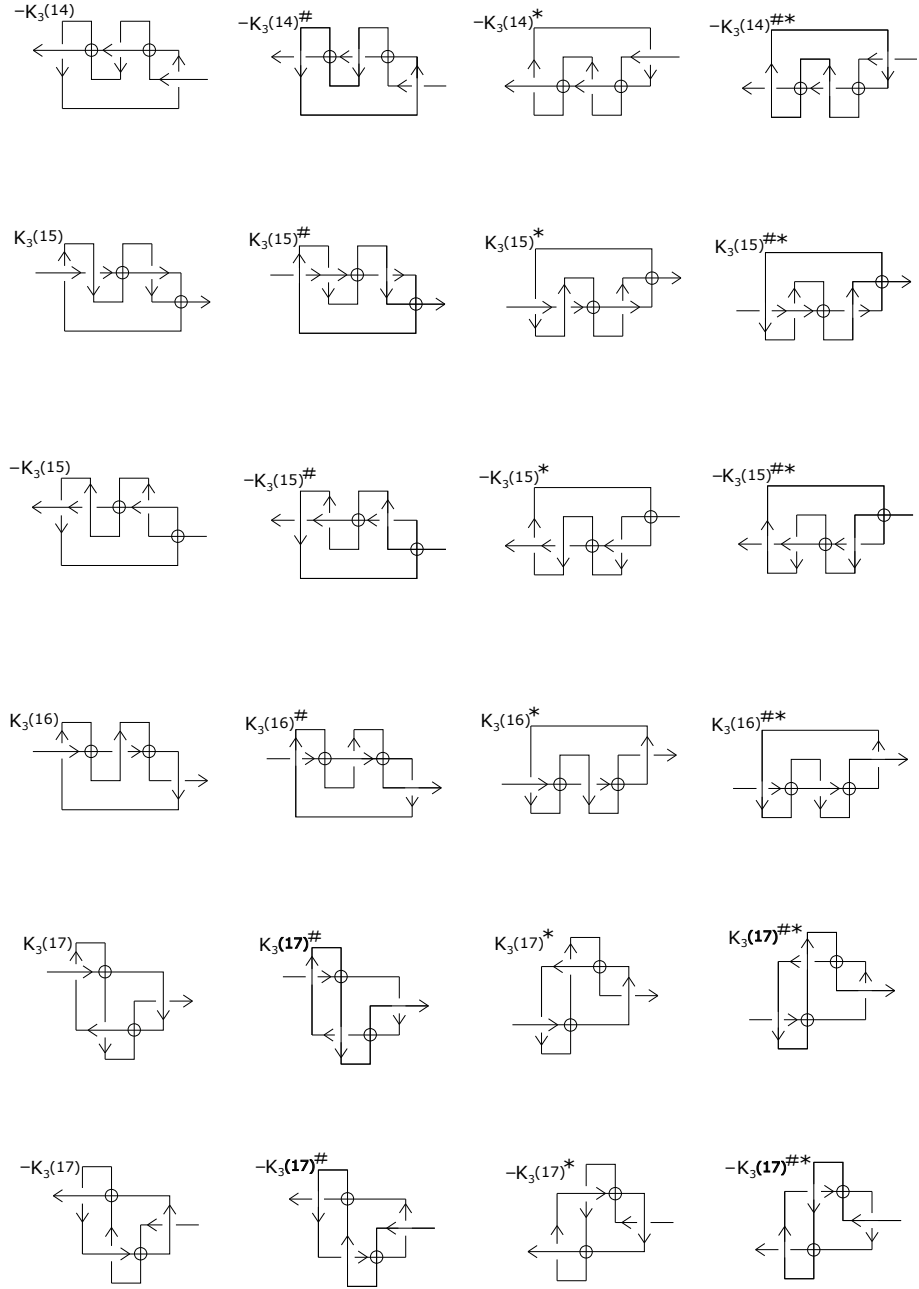}
 \end{center}
 \caption{Long virtual knots 5/10}
 \label{computation_symmetry5}
\end{figure}

\begin{figure}[htbp]
 \begin{center}
  \includegraphics[width=120mm]{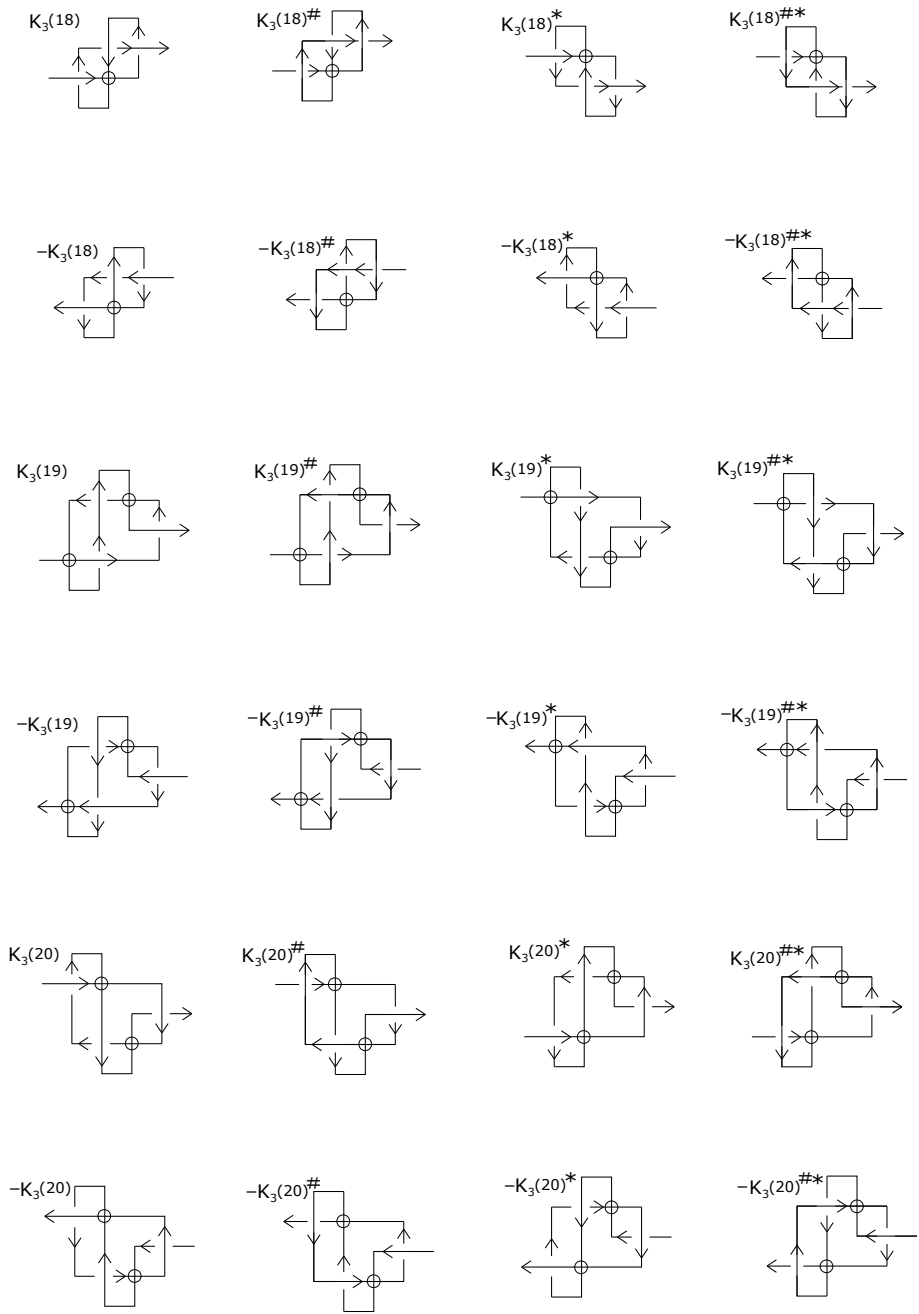}
 \end{center}
 \caption{Long virtual knots 6/10}
 \label{computation_symmetry6}
\end{figure}

\begin{figure}[htbp]
 \begin{center}
  \includegraphics[width=120mm]{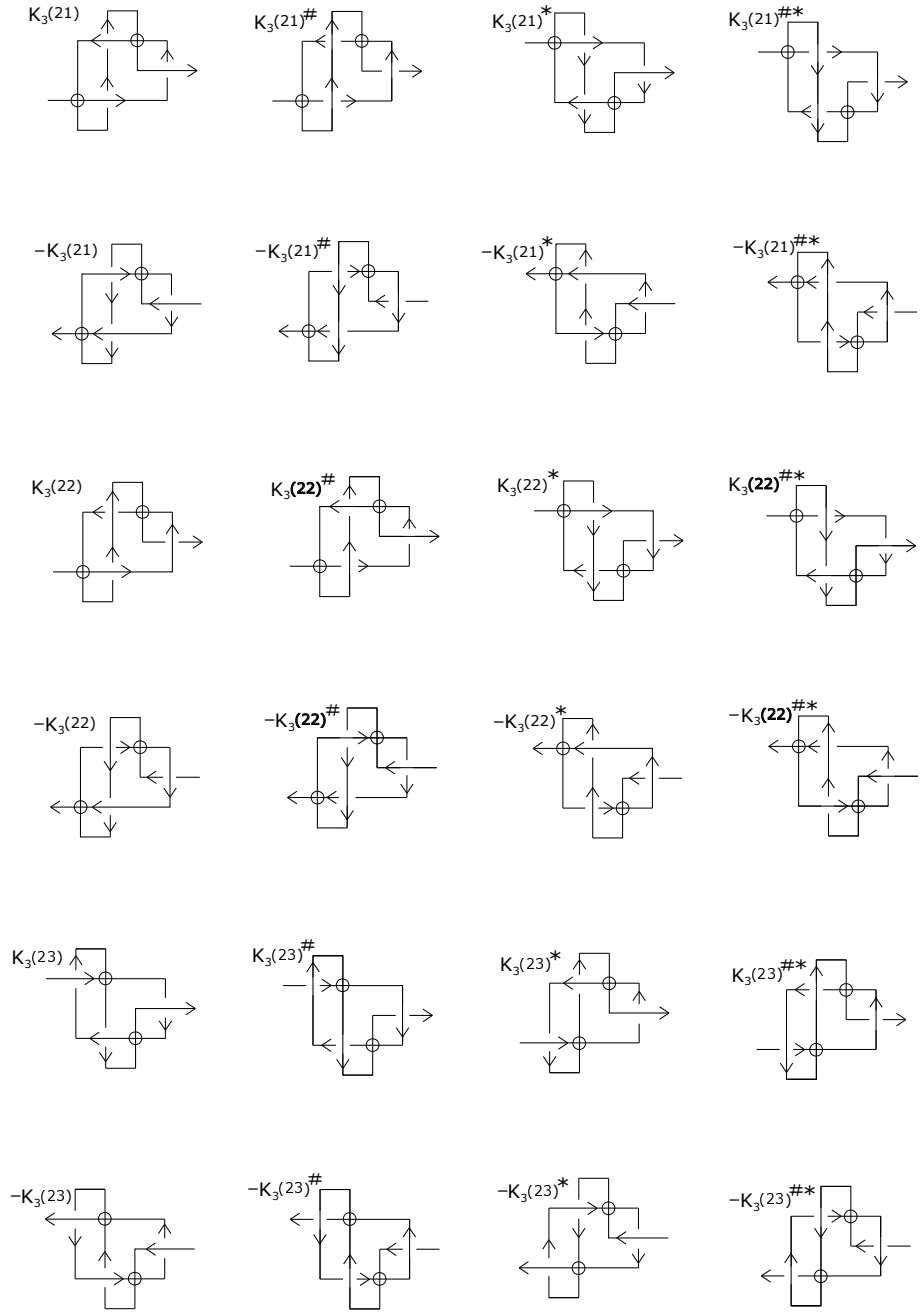}
 \end{center}
 \caption{Long virtual knots 7/10}
 \label{computation_symmetry7}
\end{figure}

\begin{figure}[htbp]
 \begin{center}
  \includegraphics[width=120mm]{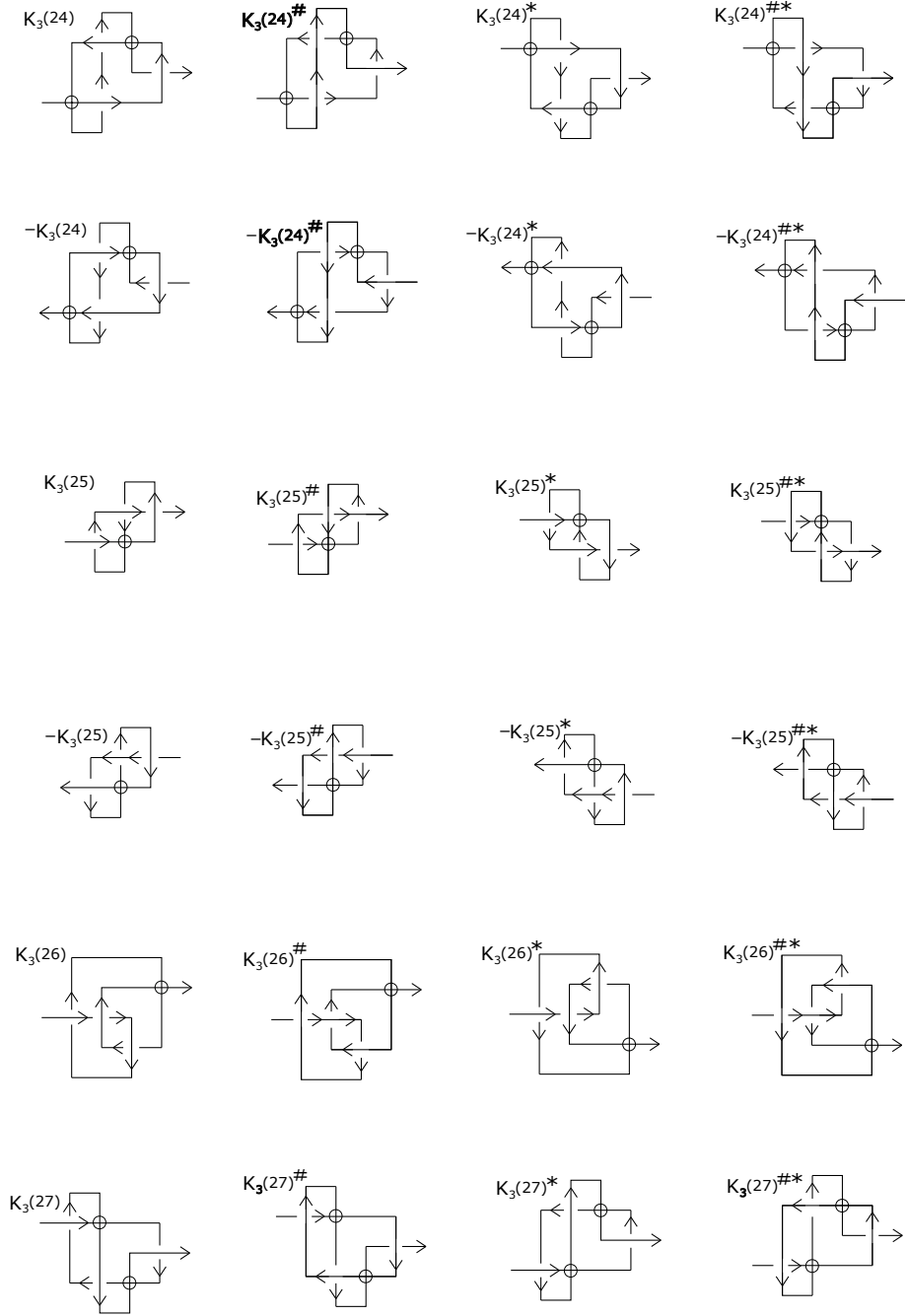}
 \end{center}
 \caption{Long virtual knots 8/10}
 \label{computation_symmetry8}
\end{figure}

\begin{figure}[htbp]
 \begin{center}
  \includegraphics[width=120mm]{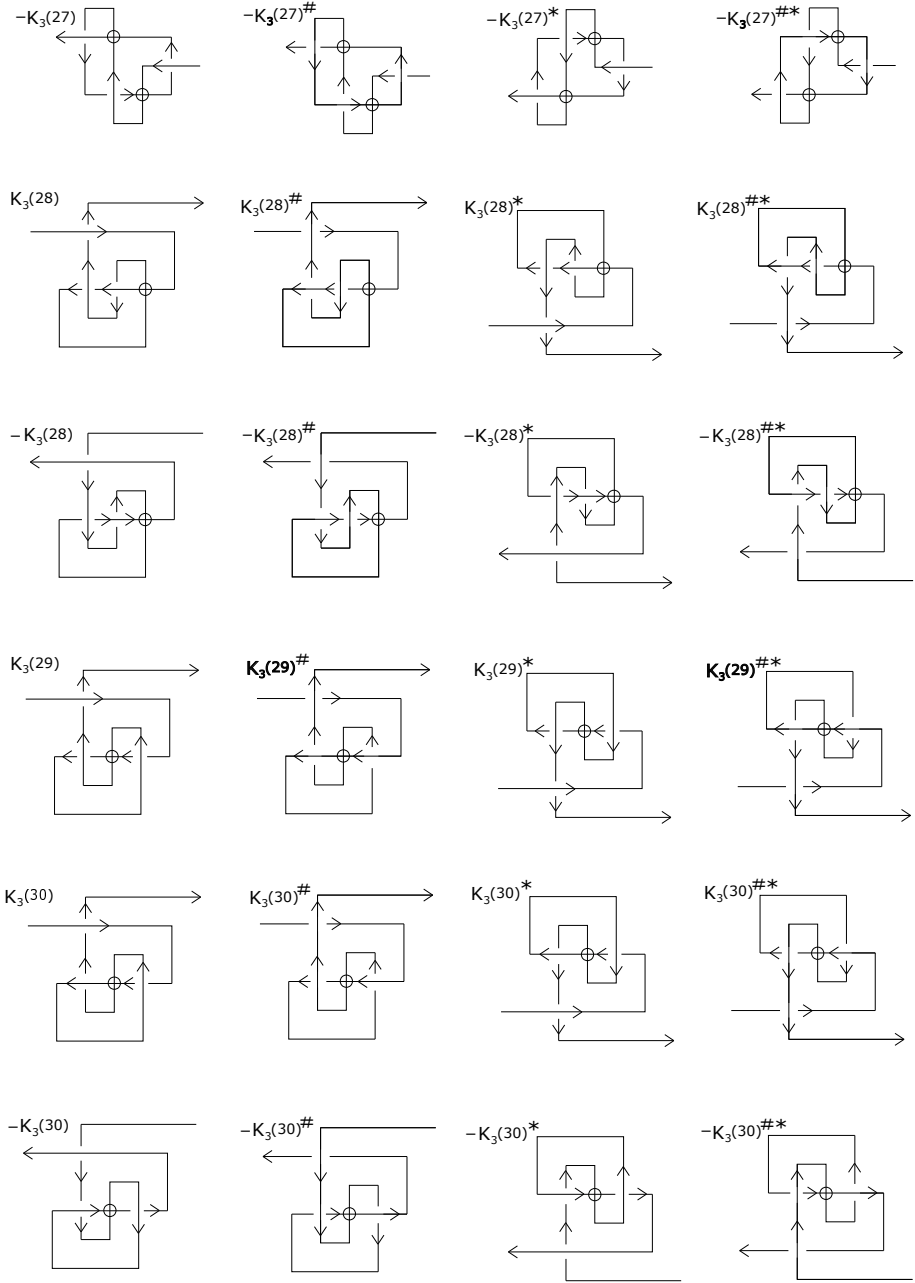}
 \end{center}
 \caption{Long virtual knots 9/10}
 \label{computation_symmetry9}
\end{figure}

\begin{figure}[htbp]
 \begin{center}
  \includegraphics[width=120mm]{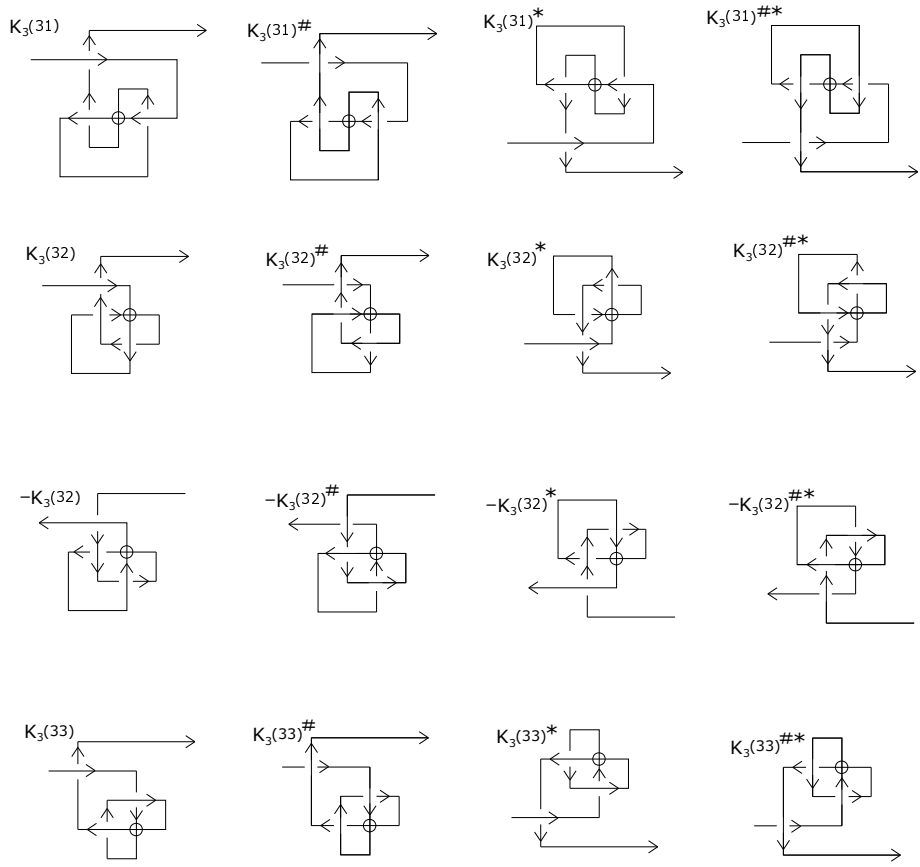}
 \end{center}
 \caption{Long virtual knots 10/10}
 \label{computation_symmetry10}
\end{figure}


\begin{thebibliography}{99}
\bibitem{yoshida} T. Yoshida. About long virtual knots whose crossing numbers are up to three (written in Japanese). Reports of the workshop ``Mathematics of Knots VI", 228--234.
\end{thebibliography}
\end{document}